\documentclass[11pt]{amsart}
\usepackage{eurosym}
\usepackage{amsfonts}
\usepackage{amssymb}
\usepackage{graphicx}
\usepackage{amsmath,amsfonts,amsmath,amssymb}
\usepackage{mathrsfs}
\usepackage[latin1]{inputenc}
\usepackage[T1]{fontenc}
\usepackage{color,xspace}
\usepackage[usenames,dvipsnames]{pstricks}
\newtheorem{theorem}{Theorem}[section]

\newtheorem{corollary}[theorem]{Corollary}

\newtheorem{lemma}[theorem]{Lemma}

\newcommand{\R}{\mathbb{R}}

\newcommand{\p}{\partial}

\newcommand{\dive}{\mathrm{div}}

\newcommand{\norm}[1]{\left\Vert#1\right\Vert}

\newcommand{\abs}[1]{\left\vert#1\right\vert}
\newcommand{\set}[1]{\left\{#1\right\}}
\newcommand{\para}[1]{\left(#1\right)}

\newcommand{\To}{\longrightarrow}

\newcommand{\uu}{\mathbf{u}}
\newcommand{\vv}{\mathbf{v}}

\newcommand{\ff}{\mathbf{f}}
\newcommand{\n}{N_{s,\varphi}}

\usepackage{times}

\usepackage{epsfig}

\begin{document}

\title[Biot consilidation model in poro-elasticity]{Carleman estimate for Biot consolidation system in poro-elasticity and application to inverse problems}

\author{M. Bellassoued}
\address{University of Carthage, Faculty of Sciences of Bizerte, Dep. of Mathematics, 7021 Jarzouna, Bizerte, Tunisie}
\email{mourad.bellassoued@fsb.rnu.tn}
\author{B. Riahi}
\address{University of Carthage, Faculty of Sciences of Bizerte, Dep. of Mathematics, 7021 Jarzouna, Bizerte, Tunisie}
\email{riahi.bochra@gmail.com}


\date{\today}
\maketitle
\begin{abstract}
In this paper, we consider a coupled system of mixed hyperbolic-parabolic type which describes the Biot consolidation model
in poro-elasticity. We establish a local Carleman estimate for Biot consilidation system. 
Using this estimate, we prove the uniqueness and a H\"{o}lder stability in determining on the one hand a physical parameter arising in connection with secondary consolidation effects $\lambda^*$ and on the other hand the two spatially varying density by a single measurement of solution over $\omega \times (0, T)$, where $T>0$ is a sufficiently large time
and a suitable subbdomain $\omega$ satisfying  $\p \omega \supset \p \Omega$.
\end{abstract}

\section{Introduction}

Let us consider an open and bounded domain $\Omega$ of $\mathbb{R}^3$ with $\mathcal{C}^{\infty} $ boundary
$\Gamma = \partial \Omega$.
Given $T>0$, the Biot consolidation model in poro-elasticity in which we are interested is the following

\begin{equation}\label{1.1}
\left\{
\begin{array}{ll}
\uu_{tt} - \Delta_{\mu, \lambda}\uu -  \nabla(\lambda^*(x) \dive \uu_t) + \varrho_1(x) \nabla \theta = \ff  &\,\,
\textrm{in}\,\, Q\equiv\Omega\times(0,T),\cr
\theta_t - \Delta \theta + \varrho_2(x) \dive \uu_t= g &\,\,
 \textrm{in}\,\, Q,
\end{array}
\right.
\end{equation}
with Dirichlet boundary condition
\begin{equation}\label{1.2}
\uu(x,t)= 0,\quad \theta(x,t)= 0 ,\quad \textrm{on} \,\, \Sigma\equiv\Gamma\times(0,T),
\end{equation}
where the $._t$ stands for the time derivative,
$\nabla = (\partial_1, \partial_2, \partial_3)$,
 and $\Delta_{\mu, \lambda}$ is the elliptic second order linear differential operator given by
\begin{eqnarray}\label{1.3}
\Delta_{\mu, \lambda} \vv(x) &\equiv& \mu \Delta \vv(x)+ (\mu + \lambda)( \nabla (\dive \vv(x)))\cr
&+& \dive \vv(x) \nabla\lambda(x) + (\nabla \vv + (\nabla \vv)^T) \nabla\mu(x), \quad x\in \Omega,
\end{eqnarray}
for $ \vv = (v_1, v_2, v_3)^T$, where $.^T$ denotes the transpose of matrices. Throughout this paper,
$t$ and $x= (x_1, x_2, x_3)$ denote the time variable and the spatial variable respectively, and
$ \uu = (u_1, u_2, u_3)^T$ denotes the displacement at the location $x$ and the time $t$, and $\theta \equiv \theta(x,t)$,
 the temperature, is a scalar function, $g$ is a heat source.
\medskip
We will assume that the Lam\'e parameters $\mu, \lambda \in \mathcal{C}^2(\overline{\Omega}
)$, satisfy
\begin{equation*}
 \mu(x)> 0, \quad \lambda(x) +  2\mu(x) > 0, \quad \forall x\in \overline{\Omega}.
\end{equation*}

We can prove (e.g., \cite{[HMP],[LM]}) that the system (\ref{1.1}) possesses a unique solution $(\uu, \theta)$ with suitable initial values.\\

The main subject of this paper is the inverse problem of determining, not only, a physical parameter arising in connection with secondary consolidation effects $\lambda^*$ but also, the two spatially varying density $\varrho_1 $ and $\varrho_2$, in the Biot consolidation model in poro-elasticity,  uniquely from observed data of displacement vector $u$ and the temperature $\theta$ on a suitable subdomain $\omega \subset \Omega$ and the observation data of $u$ and $\theta$ at given a suitable time $t_0$. Such kinds of observation data are similar to those considered in (e.g. \cite{[B&Y]}, \cite{[BW2011]}), which are typical for obtaining the corresponding stability results by the Carleman estimates. 
\subsection{Inverse Problem}

Let $\omega \subset \Omega$ be a given arbitrarily subdomain such that $\p \omega \supset \Gamma$. The condition $\p \omega \supset \Gamma$, is used to deal with the lack of a $\dive \uu$ boundary condition.
More precisely, to overcome the difficulty due to the strong coupling in the system (\ref{1.1}), we have to find the Carleman estimate for divergence of the first equation of the system (\ref{1.1}). But we do not know the boundary data of $\dive \uu$. Due to this reason, we convert the estimates of $\dive \uu$
on the boundary $\Gamma$ to the estimates in $\omega$, a neighborhood of the whole boundary. \\
The goal of this paper is to derive a H\"{o}lder stability and the uniqueness, by measurements
 $$\uu\mid_{\omega \times (0, T)}, \quad \uu(x, t_0),\quad \text{and}\quad \theta(x, t_0) \qquad x\in \Omega.$$
The key ingredient in our argument is an $L^2$-weighted inequality of Carleman type for consolidation Biot's system. This coupled mixed hyperbolic-parabolic system can describe the phenomena arising when a soil is submitted
to a load as well as the ultrasonic propagation in fluid-saturated parous media like cancellous bone.
The displacement vector field of the system, denoted by, $u$, satisfies the conservation of momentum
while the fluid pressure, $\theta$, satisfies a diffusion equation.
\medskip

K.Terzaghi \cite{[T]} was the first in interesting in the consolidation phenomena arising in porous media
under a load. He showed  the similarities between this phenomena and the exit of a flow out of a porous media, that contributed to
model fluid flows in saturated deformable porous media as a coupled flow-deformation process.
\medskip

Later, M. A. Biot \cite{[Biot1], [Biot2], [Biot3], [Biot4], [Biot5]} studied these problems assuming that the continuum mechanics laws
are applicable. He develop thus the now classical theory of poro-elasticity and proved that the linear theory of consolidation
could be established by using the Darcy law for laminar flows combined with the momentum balance equations with Hooke law for elastic
deformations.
\medskip

In the particular case when the secondary consolidation term $\lambda^* = 0$ in system (\ref{1.1}), we are interested in the thermoelastic model:
\begin{equation}\label{1}
\left\{
\begin{array}{ll}
\rho u_{tt}(x,t)- \Delta_{\mu, \lambda}u(x,t) + \varrho \nabla \theta(x,t)= \ff(x,t)  & \textrm{in }\,\,
Q\equiv\Omega\times[0,T]\cr
c_0 \theta_t(x,t) - \Delta \theta(x,t) + \varrho \dive u_t= h(x,t) & \textrm{in }\,\,
Q
\end{array}
\right.
\end{equation}
Bellassoued and Yamamoto \cite{[BY]} established Carleman estimates with second large parameter for a coupled parabolic-hyperbolic system, a thermoelastic plate system and a thermoelasticity system with residual stress.
According to the linear theory of thermoelasticity, Bellasoued and Yamamoto \cite{[B&Y]} consider a bounded and isotropic body whose mechanical behaviour is described by the Lam\'e system coupled with the heat equation. Assuming the null surface displacement on the whole boundary. They prove a H\"{o}lder stability estimate
for the inverse problem of determining the heat source only by observation of
surface traction on a suitable subdomain along a sufficiently
large time interval using Carleman estimate for thermoelasticity system. 
\medskip

Then, in a thermoelastic model, B.Wu and J.Liu \cite{[BW2011]} study the inverse problem of determining two spacially varying coefficients with the following observation data: displacement in a subdomain $\omega$ satisfying $\p \omega \supset \p \Omega$ along a sufficiently large time interval, both displacement and temperature
at a suitable time over the whole spatial domain. Based on a Carleman estimate on the hyperbolic-parabolic system, B.Wu and J.Liu prove the lipschitz stability and the uniqueness for this inverse problem under some a priori information.
\medskip

Our method is based on the tool of Carleman estimates, which was originally introduced in the field of coefficient inverse problems for hyperbolic,
parabolic and elliptic equations with the lateral data by Bukhgeim and Klibanov \cite{[BK]}.
\medskip

For the formulation of inverse problems with a finite number of
observations, Bukhgeim and Klibanov \cite{[BK]}
proposed a remarkable method based on a
Carleman estimate and established the uniqueness for
inverse problems of determining spatially varying
coefficients for scalar partial differential
equations.  See also Bellassoued and Yamamoto \cite{[BY]},
\cite{[B&Y]},  \cite{[B&Y2]} Bukhgeim \cite{[B]}, Bukhgeim,
Cheng, Isakov and Yamamoto \cite{[BCIY]},  Imanuvilov and Yamamoto
\cite{[IY1]}, Isakov 
\cite{[I2]}, Kha\u\i darov \cite{[KH1]}, Klibanov \cite{[K1]},
\cite{[KL]}, Klibanov and Timonov \cite{[KT]},
Klibanov and Yamamoto \cite{[KY]}.

%
\medskip

However, to the best authors's knowledge, the inverse problem for the Biot consolidation model in poro-elasticity, especially in the cases of multiple
coefficients, have not been studied thoroughly yet. The main difficulties for these inverse problems come from:
\begin{itemize}
\item[(i)]  In the Carleman estimates we should choose the same weight function to deal with equations of different types,
namely parabolic equation and a strongly damped hyperbolic equation.
\item[(ii)] The interaction between the two equations due to the strong coupling of displacement and temperature requires much more complicated mathematical analysis.
\end{itemize}

There are not many works concerning Carleman estimates for strongly coupled systems of partial differential equations where the
principal parts are coupled.

\subsection{Notations and statement of main results}
In order to formulate our results, we need to
introduce the following assumptions.\\
{\bf Assumption (A.1): } Let $x_0\in\R^3\backslash\overline
{\Omega}$.  Then we introduce the conditions on the scalar functions $\mu$ and $2\mu+\lambda$:
\begin{equation}\label{e1}
\mu, \lambda \in \mathcal{C}^2(\overline{\Omega}), \quad
\mu(x)\geq \mu_0>0,\qquad
2\mu(x)+\lambda(x)\geq \mu_1\quad
x\in\overline{\Omega}
\end{equation}
and there exist $r_0\in (0,\mu_0)$  and $r_1\in (0,\mu_1)$ such that
\begin{equation}\label{R1}
\frac{3}{2}\abs{\nabla\para{\log \mu}}\abs{x-x_0}\leq 1-\frac{r_0}{\mu_0},\qquad \frac{3}{2}\abs{\nabla\para{\log (2\mu+\lambda)}}\abs{x-x_0}\leq 1-\frac{r_1}{\mu_1} \quad x\in\overline{\Omega}.
\end{equation}
{\bf Assumption (A.2): }
We assume that the solution $(\uu,\theta)$ satisfies the a priori boundedeness:
\begin{equation}\label{1.5}
\norm{\uu}_{W^{7,\infty}(Q)}+\norm{\theta}_{W^{5,\infty}(Q)}\leq M_0,
\end{equation}
for some given positive constant $M_0$.\\
Throughout this paper, let consider the admissible set for fixed sufficiently smooth functions $\alpha_1$ and $\alpha_2$ on $\Gamma$
\begin{equation}\label{e2}
\mathcal{A} = \{\lambda^*\in \mathcal{C}^2(\overline{\Omega}), \norm{ \lambda^*}_{\mathcal{C}^2 (\overline{\Omega})}^2\leq m, \lambda^* = \alpha_1, \nabla \lambda^*= \alpha_2 \; \text{on}\; \Gamma,  \lambda^*(x) > \lambda_0 >0 \;\text{in}\; \Omega\}
\end{equation}
for fixed $m>0$. Moreover, for fixed sufficiently smooth functions $\alpha_0$ and $\alpha$ on $\Gamma$ and positive constants $M, \varrho_0$ and $\epsilon$, we set
\begin{multline}\label{e0}
\mathcal{B} = \{(\varrho_1, \varrho_2)\in (\mathcal{C}^2(\overline{\Omega}))^2, \norm{ \varrho_1}_{\mathcal{C}^2 (\overline{\Omega})}^2+ \norm{ \varrho_2}_{\mathcal{C}^2 (\overline{\Omega})}^2 \leq M,\cr
\varrho_1 = \alpha_0, \nabla \varrho_1 = \alpha \; \text{on}\; \Gamma,
\varrho_1(x) > \varrho_0 >0\; \text{in} \; \omega\}.
\end{multline}
For $(\vv,y)\in (H^{8}(Q))^3\times H^3(Q)$, we denote the following norms at times $t_0$ by:
\begin{equation}\label{**}
N_{t_0}(\vv, y) =\|y(\cdot,t_0)\|^2_{H^2(\Omega)}+\sum_{j=0}^{3}\|\dive\p_t^j\vv(\cdot, t_0)\|^2_{H^2(\Omega)} 
\end{equation}
and
\begin{equation}\label{****}
M_{t_0}(\vv, y) =  \norm{y (\cdot, t_0)}_{H^3(\Omega)}^{2}+\sum_{j=1}^4 \norm{\dive\p_t^i\vv(\cdot, t_0)}_{H^3(\Omega)}^{2} 
\end{equation}
Finally, let $t_0\in(0,T)$ such that
\begin{equation}\label{1.7}
\frac{1}{\sqrt{r_0}}\max_{x\in\overline{\Omega}}
\abs{x-x_0} < \min\{t_0,T-t_0\}.
\end{equation}
Throughout this paper, we assume that $\mu$ and $2 \mu + \lambda$ satisfy (\ref{e1}) and  $(\uu, \theta)$ satisfies the a priori boundedeness (\ref{1.5}).
\begin{theorem}\label{T.1}
We assume that
$$ \text{and} \qquad \abs{\dive \uu_t(x, t_0) }\geq \epsilon.$$
Let $\lambda^*, \widetilde{\lambda}^* \in \mathcal{A}$, $t_0\in(0,T)$ such that (\ref{1.7})
and let assumptions $(A.1)$ and $(A.2)$ be held.
Then, there exist constants
$C=C(\Omega,T,t_0,M_0)>0$ and $\kappa\in(0,1)$
such that
\begin{equation}\label{1.10}
\| \lambda^*-\widetilde{\lambda}^*  \|^2_{H^2(\Omega)}
\leq C\para{\|\uu-\widetilde{\uu}\|^2_{H^4(\omega\times(0,T))} + N_{t_0}(\uu-\widetilde{\uu}, \theta-\widetilde{\theta})}^\kappa.
\end{equation}
\end{theorem}
By Theorem \ref{T.1}, we can readily derive the uniqueness in the inverse problem
\begin{corollary}\label{c1}
Under the same assumptions as in Theorem \ref{T.1} and if
$$ \uu (x, t) = \widetilde{\uu} (x, t),\qquad  (x, t) \in \omega \times (0, T),$$
$$ \uu (., t_0) = \widetilde{\uu} (., t_0),\; \theta (., t_0) = \widetilde{\theta} (., t_0) \qquad  x \in \Omega .$$
Then, $\lambda^* = \widetilde{\lambda}^*, \qquad \text{in}\quad \Omega$.
\end{corollary}

\begin{theorem}\label{T.2}
We assume that
$$ \abs{\nabla \theta(x, t_0) \cdot (x - x_0)}\geq \epsilon. $$
Let $(\varrho_1, \varrho_2), (\widetilde{\varrho_1}, \widetilde{\varrho_2}) \in \mathcal{B}$, $t_0\in(0,T)$ such that (\ref{1.7})
and let assumptions $(A.1)$ and $(A.2)$ be held.
Then, there exist constants
$C=C(\Omega,T,t_0,M_0)>0$ and $\delta\in(0,1)$
such that
\begin{equation}\label{2}
\norm{\widetilde{\varrho}_1 - \varrho_1 }^2_{H^1(\Omega)} + \norm{\widetilde{\varrho}_2 - \varrho_2}^2_{L^2(\Omega)}
\leq C\para{\norm{\widetilde{\uu}-\uu}^2_{H^7(\omega\times(0,T))} + M_{t_0}(\widetilde{\uu}-\uu, \widetilde{\theta}-\theta)}^{\delta}.
\end{equation}
\end{theorem}
By Theorem \ref{T.2}, we can readily derive the uniqueness in the inverse problem
\begin{corollary}\label{c2}
Under the same assumptions as in Theorem \ref{T.2} and if
$$ \uu (x, t) = \widetilde{\uu} (x, t),\qquad  (x, t) \in \omega \times (0, T),$$
$$ \uu (., t_0) = \widetilde{\uu} (., t_0),\; \theta (., t_0) = \widetilde{\theta} (., t_0) \qquad  x \in \Omega .$$
Then, $\varrho_1 = \widetilde{\varrho_1}$ and $\varrho_2 = \widetilde{\varrho_2}$ in $\Omega$.
\end{corollary}

The remainder of the paper is organized as follows. In section $2$, we give a Carleman estimate for Biot's consolidation system. In section $3$, we prove
Theorem \ref{T.1} and
Theorem \ref{T.2}.

\section{Carleman estimate for Biot's system}
\setcounter{equation}{0}
In this section we will prove Carleman estimates
for the Biot's consolidation system. In order to formulate
our Carleman type estimates we introduce some
notations.
Let $\vartheta:\overline{\Omega}\To\R$ be the
strictly convex function given by
\begin{equation}\label{2.1}
\vartheta(x)=\abs{x-x_0}^2,\quad
x\in\overline{\Omega},
\end{equation}
where $x_0\notin\overline{\Omega}$.
Let us define $t_0$ by
\begin{equation}\label{2.2}
\min\{ t_0^2,(T-t_0)^2 \}
>r_0^{-1}
\para{\max_{x\in\overline{\Omega}}\vartheta(x)},
\end{equation}
and let
\begin{equation}\label{2.3}
\psi(x,t)=\vartheta(x)-\beta\para{\para{t-t_0}^2-M}.
\end{equation}
Fix $\delta>0$ and $\beta>0$ such that
\begin{equation}\label{2.4}
\beta\min\{t_0^2,(T-t_0)^2\}
> \max_{x\in\overline{\Omega}}\vartheta(x)
+ \delta,\quad 0<\beta<r_0.
\end{equation}
Then the function $\psi(x,t)$ verifies
the following properties
\begin{equation}\label{2.5}
\psi(x,0)\leq \beta M-\delta,\quad \psi(x,T)
\leq \beta M-\delta,\quad \textrm{for all}
\quad x\in\overline{\Omega},
\end{equation}
there exists $\epsilon\in (0,T/4)$ such that
\begin{equation}\label{2.6}
\max_{x\in\overline{\Omega}}\psi(x,t)\leq \beta M-\frac{\delta}{2},\quad \textrm{for all}\quad t\in (0,2\epsilon)\cup (T-2\epsilon,T),
\end{equation}
and
\begin{equation}\label{2.7}
\min_{x\in\overline{\Omega}}\psi(x,t_0)\geq \beta M.
\end{equation}
We define now the weight function $\varphi:
\overline{\Omega}\times\R\To\R$ by
\begin{equation}\label{2.8}
\varphi(x,t)=e^{\gamma\psi(x,t)},\quad \gamma>0,
\end{equation}
where $\gamma$ is a large parameter selected in the following and let
\begin{equation}\label{2.9}
\sigma\equiv \sigma(x,t)=s\gamma\varphi(x,t).
\end{equation}
We use usual function spaces, $H^k(Q)$, and
$$
H^{1,2}(Q)=H^1(0,T;L^2(\Omega))\cap L^2(0,T;H^2(\Omega)).
$$
Let $(\vv,y)$  a solution of the linear Biot consolidation system
\begin{equation}\label{2.10}
\left\{
\begin{array}{ll}
\vv_{tt}(x,t)- \Delta_{\mu, \lambda}\vv(x,t) -  \nabla(\lambda^* \dive \vv_t(x,t)) + \varrho_1 \nabla y(x,t)= \ff(x,t)  & \textrm{in }\,\,
Q\equiv\Omega\times(0,T)\cr
y_t(x,t) - \Delta y(x,t) + \varrho_2 \dive \vv_t(x,t)= h(x,t) & \textrm{in }\,\,
Q
\end{array}
\right.
\end{equation}
such that
\begin{equation}\label{2.12}
\begin{array}{ll}
\textrm{Supp}(\vv(\cdot,t))\subset\Omega,\quad \textrm{Supp}(y(\cdot,t))\subset\Omega, & \quad\textrm{for all}\quad t\in(0,T).\cr
\p_t^j\vv(x,0)=\p^j_t\vv(x,T)=0,\quad y(x,0)=y(x,T)=0 & \quad\textrm{for all}\quad x\in\Omega,\, j=0,1.
\end{array}
\end{equation}

A Carleman estimate is  an inequality for a solution to a partial differential equation with weighted $L^2$-norm and is effectively applied for proving
the unique continuation for a partial differential equation with non-analytic coefficients. As a pioneering work concerning a Carleman estimate, we can refer
to Carleman's paper \cite{[Carleman]} where what is called a Carleman estimate is proved and applied it for proving the uniqueness
in the Cauchy problem for a two-dimensional elliptic equation.
\medskip

The following theorem is a weighted Carleman estimate with second large parameter for Biot's consolidation system (\ref{2.10}) with assumption (\ref{2.12}).
\begin{theorem}[Carleman estimate for Biot's consolidation system]\label{t1}
There exist two constants $\gamma_* > 0$ and $C>0$ such that
for any $\gamma > \gamma_*$, there exists $s_* = s_*(\gamma)> 0$ such that
the following estimate holds
\begin{multline}\label{11}
C\int_Q  \Big(\sigma \abs{\nabla_{x,t} \vv}^2
+  \sigma^{3}\abs{\vv}^2
+ \sigma^4 \abs{ \dive \vv}^{2} + \sigma^3 \abs{ \dive \vv_t}^{2} + \sigma^2 \abs{  \nabla\dive \vv}^{2}  + \sigma \abs{ \nabla\dive \vv_t}^{2}\cr
+\abs{\Delta y}^2
+ \sigma^2 \abs{\nabla y}^2
+ \sigma^{4}\abs{y}^2 \Big)e^{2s\varphi} dx\,dt
\leq \int_Q (\abs{\ff}^2 + \abs{\nabla \ff}^2 + \gamma^{-1}\sigma \abs{h}^2 ) e^{2s\varphi}
dx\,dt\cr
+ \int_{Q}\para{ \gamma^{-1}\abs{\Delta \dive \vv(x,t_0)}^2
+ \gamma^{-1}\abs{\dive \vv_t(x,t_0)}^2+\sigma^4\abs{\dive \vv(x,t_0)}^2
+ \sigma^2\abs{\nabla \dive \vv(x,t_0)}^2}e^{2s\varphi}dx dt
\end{multline}
for any solution $(\vv, y) \in H^2(Q) \times H^{2,1}(Q)$ to problem (\ref{2.10}) which satisfy (\ref{2.12}) and any $s \geq s_*$.
\end{theorem}
In order to prove Theorem \ref{t1}, we need a parabolic, strongly damped hyperbolic and hyperbolic Carleman estimates with second large parameter.
\subsection{Carleman estimate for parabolic equation}
We consider the second order parabolic operator $\p_t - \Delta$. As for Carleman estimate for parabolic
equations with singular weight function, we can refer to Fursikov and Imanuvilov \cite{[FI]}, Imanuvilov \cite{[I]}, Imanuvilov and Yamamoto \cite{[IY2]}.
Here we give a Carleman estimate for parabolic equations with the regular weight function $\varphi$.
\medskip

The following parabolic Carleman estimates holds
\begin{lemma}\label{L4.2}
Let  $k\in \mathbb{N}\cup\set{0}$.
There exist three positive constants $\gamma_*$, $s_*$ and $C$ such that,
for any $\gamma\geq \gamma_*$ and any $s\geq s_*$, the following inequality
holds:
\begin{multline}\label{4.2}
C\gamma \int_Q \Big(\sigma^{k-1}\sum_{\abs{\alpha}=2}\abs{\p^\alpha y}^2
+ \sigma^{k+1}\abs{\nabla y}^2
+ \sigma^{k+3}\abs{y}^2\Big)e^{2s\varphi} dx\,dt\cr
\leq \int_Q\sigma^k\abs{\para{y_t-\dive(\lambda^*\nabla y)}}^2e^{2s\varphi}
dx\,dt
\end{multline}
for any $y\in H^{1,2}(Q)$ with compact support in $Q$.
\end{lemma}
 As for the proof, we can refer to Bellassoued and Yamamoto  \cite{[BY]}.
\subsection{Carleman estimate for strongly damped hyperbolic equation}

In this subsection we will prove Carleman estimate for a strongly damped hyperbolic equation with variable coefficients. We consider the following hyperbolic equation
\begin{equation}\label{4.13}
 \p_t^2v-(2\mu+\lambda)\Delta v - \dive(\lambda^*(x)\nabla \p_t v) = f
\qquad \textrm{in}\quad Q\equiv \Omega \times (0,T).
\end{equation}
Many physical phenomena are properly described by the strongly damped hyperbolic equation as (\ref{4.13})
such as equation  of this type can be considered as a class of linear evolution equations governing the motion of
a viscoelastic solid (for example, a bar if the space is $\R$ and a plate if the space is $\R^2$ ) composed of the material of the rate type.\\
The following Theorem is a weighted Carleman estimate with a second
larger parameter  for the strongly damped wave equation (\ref{4.13}).
\begin{theorem}\label{T.3}
There exist two constants $\gamma_* > 0$ and $C>0$ such that
for any $\gamma > \gamma_*$, there exists $s_* = s_*(\gamma)> 0$ such that
the following estimate holds
\begin{multline}\label{4.15}
C\!\!\int_Q\Big(\sigma^2\abs{\nabla v}^2
+ \sigma^4\abs{v}^2+\sigma\abs{\nabla v_t}^2+\sigma^3\abs{v_t}^2\Big)
e^{2s\varphi}dxdt
\leq \gamma^{-1} \int_Q\abs{ f}^2
e^{2s\varphi}dxdt\cr
+ \int_{Q}\para{\gamma^{-1}\abs{\Delta v(x,t_0)}^2
+ \gamma^{-1}\abs{v_t(x,t_0)}^2+\sigma^4\abs{v(x,t_0)}^2
+ \sigma^2\abs{\nabla v(x,t_0)}^2}e^{2s\varphi}dx dt
\end{multline}
for any solution $v \in H^3(Q)$ to problem (\ref{4.13}) compactly supported in $Q$ and any $s\geq s_*$.
\end{theorem}
In order to prove Theorem \ref{T.3}, we need the following two Lemmas:
\begin{lemma}\label{L4.4}
For any $k\in\R_+$, the following estimate holds
$$
\int_Q\sigma^{2k+1} \abs{\int_{t_0}^t w(x,\tau)d\tau}^2 e^{2s\varphi} dx dt\leq C \int_Q\sigma^{2k} \abs{w(x,t)}^2 e^{2s\varphi} dxdt,
$$
for any $w\in L^2(Q)$.
\end{lemma}
\begin{proof}
By the Cauchy schwartz inequalities, we obtain
\begin{equation*}
 \int_Q \sigma e^{2 s \varphi} \abs{\int_{t_0}^{t} w(x, \tau) d\tau}^2 dxdt \leq
  \int_Q  \sigma (t - t_0)e^{2 s \varphi} \Big(\int_{t_0}^{t} \abs{ w(x, \tau) }^2\,d\tau\Big)\, dxdt.
 \end{equation*}
Using the fact that,  
\begin{equation*}
\frac{\p}{\p t} (e^{2 s \varphi}) = - 4 \beta \sigma (t - t_0) e^{2 s \varphi},
\end{equation*}
we get,
\begin{multline*}
 \int_Q  \sigma (t - t_0)e^{2 s \varphi} \Big(\int_{t_0}^{t} \abs{ w(x, \tau) }^2d\tau\Big)\, dxdt
 = -\frac{1}{ 4 \beta} \int_Q \frac{\p}{\p t}(e^{2s \varphi}) \Big(\int_{t_0}^{t} \abs{ w(x, \tau) }^2d\tau\Big)\, dxdt\cr
 = -\frac{1}{ 4 \beta} \int_\Omega \Big(e^{2s \varphi(x,T)} \Big(\int_{t_0}^{T} \abs{ w(x, \tau) }^2d\tau\Big) + e^{2s \varphi(x, 0)} \Big(\int^{t_0}_{0} \abs{ w(x, \tau) }^2d\tau\Big)\Big)\, dx \cr
  + \frac{1}{ 4 \beta} \int_Q e^{2s \varphi}  \abs{ w(x, t) }^2\, dxdt\cr
 \leq \frac{1}{ 4 \beta} \int_Q e^{2s \varphi}  \abs{ w(x, t) }^2\, dxdt.
\end{multline*}
Then
\begin{equation}\label{1.8}
\int_Q \sigma e^{2 s \varphi} \abs{\int_{t_0}^{t} w(x, \tau) d\tau}^2 dxdt \leq
C \int_Q e^{2s \varphi}  \abs{ w(x, t) }^2\, dxdt.
\end{equation}

Let $k>0$, and let $\widetilde{w}(x, t) = \sigma^k (x,t) w(x, t) $, by (\ref{1.8}), we obtain
\begin{equation*}
\int_Q \sigma e^{2 s \varphi} \abs{\int_{t_0}^{t} \sigma^k (x, \tau) w(x, \tau) d\tau}^2 dxdt \leq
C \int_Q e^{2s \varphi}  \abs{\sigma^k (x, t) w(x, t) }^2\, dxdt.
\end{equation*}
Moreover, using the fact that $\sigma(x,t) \leq \sigma(x, \tau)$ if $t_0 \leq \tau \leq t$ or $t \leq \tau \leq t_0$, we deduce that
\begin{equation}\label{1.9}
\int_Q \sigma^{ 2 k + 1}(x, t) e^{2 s \varphi} \abs{\int_{t_0}^{t} w(x, \tau) d\tau}^2 dxdt \leq
C \int_Q \sigma^{ 2 k }(x, t) e^{2s \varphi}  \abs{ w(x, t) }^2\, dxdt.
\end{equation}
The proof is complete.
\end{proof}
\begin{lemma}\label{L5}
Let $\delta>0$, $T>0$. There exists a constant $C = (1+ \delta^2 T^2 e^{2 \delta T})>0$, such that the following estimate holds
\begin{equation*}
\int_0^ T\abs{\int_{t_0}^ t e^{\delta (\tau - t)} h(\tau) d\tau}^2 dt \leq C \int_0^T\abs{\int_{t_0}^ t h(\tau) d\tau}^2 dt,
\end{equation*}
for any $h\in L^2([0, T])$.
\end{lemma}
\begin{proof}
For $t \geq t_0$, denote 
$$
f(t) = \int_{t_0}^{t} e^{\delta \tau} h(\tau) d\tau,\quad \textrm{and}\quad g(t) = e^{- \delta t} f(t).
$$
Using the fact that, $g' (t) = h(t) - \delta g(t) $, we deduce that
\begin{equation*}
 g(t) = \int_{t_0}^ t g' (s) ds = \int_{t_0}^ t h(s) ds - \delta \int_{t_0}^ t g(s) ds .
\end{equation*}
Then,
\begin{equation*}
\abs{ g(t)} \leq \abs{\int_{t_0}^ t h(s) ds}+  \delta \int_{t_0}^ t \abs{g(s)} ds .
\end{equation*}
Applying Gr\"onwall's Lemma, we get
\begin{equation*}
\abs{ g(t)} \leq \abs{\int_{t_0}^ t h(s) ds}+ \delta\int_{t_0}^ t \abs{\int_{t_0}^\tau h(s) ds} e^{\delta (t - \tau)} d\tau.
\end{equation*}
Then,
\begin{multline*}
\int_{0}^T \abs{ g(t)}^2 dt \leq \int_0^T \abs{\int_{t_0}^ t h(s) ds}^2 dt+ \delta^2 e^{2 \delta T}\int_0^T \Big(\int_{t_0}^T \abs{\int_{t_0}^\tau h(s) ds}  d\tau\Big)^2dt \cr
\leq \int_0^T \abs{\int_{t_0}^ t h(s) ds}^2 dt+ \delta^2 T e^{2 \delta T}  \Big(\int_{t_0}^T \abs{\int_{t_0}^\tau h(s) ds}d\tau\Big)^2 \cr
\leq (1 + \delta^2 T^2 e^{2 \delta T} ) \int_0^T \abs{\int_{t_0}^ t h(s) ds}^2 dt.
\end{multline*}
For $t \leq t_0$, denote 
$$f(t) = \int_{t}^{t_0} e^{\delta \tau} h(\tau) d\tau,\quad \textrm{and}\quad g(t) = e^{- \delta t} f(t).
$$
Using the fact that, $g'(t) = - h(t) - \delta g(t) $, we deduce that
\begin{equation*}
 g(t) = -\int^{t_0}_t g'(s) ds = \int^{t_0}_t h(s) ds + \delta \int^{t_0}_t g(s) ds .
\end{equation*}
From the same argument, we conclude
\begin{equation*}
\int_{0}^T \abs{ g(t)}^2 dt \leq (1 + \delta^2 T^2 e^{2 \delta T} ) \int_0^T \abs{\int_{t_0}^ t h(s) ds}^2 dt.
\end{equation*}
The proof is complete.
\end{proof}
Now, we give the proof of Theorem \ref{T.3}.
\medskip

We denote by $P_1$ the parabolic operator given by:
\begin{equation}\label{1.10}
P_1(u) = u_t - \dive(\lambda^*(x) \nabla u).
\end{equation}
Then (\ref{4.13}) can be written as
\begin{equation}\label{1.11}
\mathcal{P}(v) :=  \p_t^2v-(2\mu+\lambda)\Delta v - \dive(\lambda^*(x)\nabla \p_t v)= P_1(v_t) - (2 \mu + \lambda) \Delta v=f
\end{equation}
We deduce that
\begin{equation}\label{1.13}
P_1(v_t) = f + (2 \mu + \lambda) \Delta v.
\end{equation}
Applying the parabolic Carleman estimate give, by Lemma \ref{L4.2} to $v_t$, we obtain
\begin{equation}\label{1.12}
C\gamma \int_Q \Big(\sigma\abs{\nabla v_t(x,t)}^2
+ \sigma^{3}\abs{v_t}^2\Big)e^{2s\varphi} dx\,dt
\leq \int_Q\abs{f}^2e^{2s\varphi}
dx\,dt + C\int_Q\abs{\Delta v}^2e^{2s\varphi}
dx\,dt
\end{equation}
Applying Lemma \ref{L4.4} to $v_t$, with $k=3/2$, we deduce
$$
\int_Q\sigma^{4} \abs{\int_{t_0}^t v_t(x,\tau)d\tau}^2 e^{2s\varphi} dx dt\leq C \int_Q\sigma^{3} \abs{v_t(x,t)}^2 e^{2s\varphi} dxdt.
$$
Then, we have
\begin{equation}\label{2.1}
\int_Q\sigma^{4} \abs{v}^2 e^{2s\varphi} dx dt\leq C \int_Q\sigma^{3} \abs{v_t}^2 e^{2s\varphi} dxdt +
C \int_Q\sigma^{4} \abs{v(x,t_0)}^2 e^{2s\varphi} dx dt.
\end{equation}
From the same argument, we conclude that
\begin{equation}\label{2.2}
\int_Q\sigma^{2} \abs{\nabla v}^2 e^{2s\varphi} dx dt\leq C \int_Q\sigma \abs{\nabla v_t}^2 e^{2s\varphi} dxdt +
C \int_Q\sigma^{2} \abs{\nabla v(x,t_0)}^2 e^{2s\varphi} dx dt.
\end{equation}
Collecting (\ref{2.1}) and (\ref{2.2}), we get
\begin{multline}\label{2.3}
\int_Q (\sigma^{4} \abs{v}^2  + \sigma^{2}\abs{\nabla v}^2) e^{2s\varphi}dx dt\leq C \int_Q (\sigma^3 \abs{v_t}^2  + \sigma\abs{\nabla v_t}^2) e^{2s\varphi}dxdt\cr
\;\; + C \int_Q (\sigma^4 \abs{v(x,t_0)}^2  + \sigma^{2}\abs{\nabla v(x,t_0)}^2)  e^{2s\varphi} dx dt.
\end{multline}
Inserting (\ref{1.12}) in (\ref{2.3}), we obtain
\begin{multline}\label{2.30}
\int_Q (\sigma^{4} \abs{v}^2  + \sigma^{2}\abs{\nabla v}^2) e^{2s\varphi} dx dt\leq C\gamma^{-1}\int_Q\abs{f}^2e^{2s\varphi}
dx\,dt \cr
+ C\gamma^{-1}\int_Q\abs{\Delta v}^2 e^{2s\varphi}
dx\,dt
 + C \int_Q e^{2s\varphi}(\sigma^4 \abs{v(x,t_0)}^2  + \sigma^{2}\abs{\nabla v(x,t_0)}^2)dx dt.
\end{multline}
Collecting (\ref{2.30}) and (\ref{1.12}), we get
\begin{multline}\label{2.11}
\int_Q e^{2s\varphi}
\Big(\sigma^{4} \abs{v}^2  + \sigma^{2}\abs{\nabla v}^2 + \sigma^{3} \abs{v_t}^2  + \sigma\abs{\nabla v_t}^2 \Big) dx dt
\leq C \gamma^{-1} \int_Q\abs{f}^2e^{2s\varphi}
dx\,dt\cr
+ C \gamma^{-1} \int_Q\abs{\Delta v}^2e^{2s\varphi}
dx\,dt + C \int_Q e^{2s\varphi}(\sigma^{2}\abs{\nabla v(x,t_0)}^2+ \sigma^4 \abs{v(x,t_0)}^2)dx dt.
\end{multline}
In order to estimate the second term in (\ref{2.11}), denote $z(x,t) = \Delta v(x,t) $, by (\ref{1.11}) $z$ satisfy the following differential equation
\begin{equation*}
 z_t + \delta(x) z =  \frac{1}{\lambda^*(x)}(v_{tt} - \nabla \lambda^* (x).\nabla v_t  - f ) := G(x,t),
\end{equation*}
with $\,\delta(x) =(2 \mu(x) + \lambda(x))/\lambda^*(x)$.\\
Then using Duhamel Formula, $z$ takes the form
\begin{equation}\label{2.13}
z(x,t)= z(x,t_0) e^{\delta(x) (t_0 - t)}
 + \displaystyle \int_{t_0}^t G(x,\tau) e^{\delta(x) (\tau - t)} d\tau.
\end{equation}
Let now compute the last integral in (\ref{2.13}). \\
We have
\begin{equation*}
\displaystyle \int_{t_0}^t G(x,\tau) e^{\delta(x) (\tau - t)} d\tau =
 \displaystyle\int_{t_0}^{t} \displaystyle\frac{1}{\lambda^*(x)} \para{\frac{\p}{\p t}(v_{t}(x, \tau) - \nabla\lambda^*(x).\nabla v(x, \tau)) - f } e^{\delta(x)  (\tau - t)}d\tau.
\end{equation*}
By integration by parts, we deduce that
\begin{multline}\label{2.4}
\displaystyle \int_{t_0}^t G(x,\tau) e^{\delta(x) (\tau - t)} d\tau =
  \displaystyle\frac{1}{\lambda^*(x)} \Big((v_{t}(x,t) - \nabla\lambda^*(x) \nabla v(x, t)) - (v_{t}(x,t_0) -  \nabla\lambda^*(x)\nabla v(x, t_0)) \Big) e^{\delta(x) ( t_0 - t)} \cr
 - \displaystyle\frac{\delta(x)}{\lambda^*(x)} \displaystyle\int_{t_0}^{t} (v_{t}(x, \tau) - \nabla\lambda^*(x).\nabla v(x, \tau)) e^{\delta(x)  (\tau - t)}d\tau
  -  \displaystyle\frac{1}{\lambda^*(x)}\displaystyle \int_{t_0}^{t}  f(x,\tau) e^{\delta(x)  (\tau - t)}d\tau .
\end{multline}
Using Lemma \ref{L5}, we deduce that
\begin{multline}\label{4}
\int_Q \abs{\Delta v}^2 e^{2 s \varphi}  dxdt \leq C\Big( \int_Q \para{\abs{ \Delta v(x,t_0)}^2 + \abs{ v_{t}(x,t_0)}^2 + \abs{\nabla v(x,t_0)}^2 } e^{2 s \varphi} dxdt \cr
  +  \int_Q  \para{\abs{ v_{t}}^2 + \abs{\nabla v}^2} e^{2 s \varphi} dxdt
 + \int_Q \abs{\int_{t_0}^{t} v_{t}(x, \tau)d\tau}^2 e^{2 s \varphi} dxdt\cr
+ \int_Q \abs{\int_{t_0}^{t} \nabla v (x, \tau)d\tau}^2 e^{2 s \varphi} dxdt + \int_Q  \abs{ \int_{t_0}^{t}  f(x,\tau)d\tau}^2 e^{2 s \varphi} dxdt\Big).
\end{multline}
Consequently, we find
\begin{multline}\label{2.6}
\int_Q e^{2s\varphi}
\Big(\sigma^{4} \abs{v}^2  + \sigma^{2}\abs{\nabla v}^2 + \sigma^{3} \abs{v_t}^2  + \sigma\abs{\nabla v_t}^2 \Big) dx dt
\leq C \gamma^{-1} \int_Q\abs{f}^2e^{2s\varphi}
dx\,dt\cr
 + C \int_Q e^{2s\varphi}(\sigma^{2}\abs{\nabla v(x,t_0)}^2+ \sigma^4 \abs{v(x,t_0)}^2  +
  \gamma^{-1}\abs{ v_t(x,t_0)}^2 + \gamma^{-1}\abs{\Delta v(x,t_0)}^2)dx dt.
\end{multline}
This completes the proof of Theorem \ref{T.3}.
\subsection{Carleman estimate for hyperbolic equation}
We recall the following Carleman estimate for a scalar hyperbolic equation. As for the proof, we
can refer to Bellassoued and Yamamoto  \cite{[BY]} and Isakov and Kim \cite{[IK08], [IK09]} which gives a direct proof by integration by parts.
\begin{lemma}\label{L3}
Let $\mu >0$ satisfy $\frac{3}{2}\abs{\nabla\para{\log \mu}}\abs{x-x_0}\leq 1.$
There exist constants $C>0$ and $\gamma_*>0$ such that for any $\gamma > \gamma_*$, there
exist $s_* = s_*(\gamma)$ such that for all  $s\geq s_*$ the following Carleman estimate
holds
\begin{equation}\label{3.0}
C\int_Q  \sigma \Big( \abs{\nabla y}^2
+ \abs{y_t}^2 + \sigma^{2}\abs{y}^2\Big)e^{2s\varphi} dx\,dt
\leq \int_Q \abs{y_{tt} - \mu(x) \Delta y}^2e^{2s\varphi}
dx\,dt,
\end{equation}
for any $y\in H^1(Q)$ with compact support in $Q$.
\end{lemma}
\subsection{Completion of Carleman estimate for Biot's system}
Let $(\vv, y)$ be a solution of Biot's model in poro-elasticity (\ref{2.10}) satisfying (\ref{2.12}).
\medskip

Let $v(x,t) =  \dive \,\vv(x,t)$, we apply $\dive$ to the first equation in (\ref{2.10}), we can derive the following equation
\begin{equation}\label{3.1}
 v_{tt}-(2\mu+\lambda)\Delta v- \dive(\lambda^*\nabla  v_t) - \nabla \lambda^*. \nabla v_t + \varrho_1 \Delta y + \nabla\varrho_1\cdot\nabla y = \dive\, \ff
\qquad \textrm{in}\quad Q.
\end{equation}
Applying Theorem \ref{T.3} to (\ref{3.1}),
\begin{multline}\label{3.2}
C \int_Q\Big(\sigma^2\abs{\nabla v}^2
+ \sigma^4\abs{v}^2+\sigma\abs{\nabla v_t}^2+\sigma^3\abs{v_t}^2\Big)
e^{2s\varphi}dxdt
\leq \gamma^{-1}\int_Q\abs{\dive\, \ff}^2
e^{2s\varphi}dxdt\cr
+ \int_{Q}\para{ \gamma^{-1}\abs{\Delta v(x,t_0)}^2
+ \gamma^{-1}\abs{v_t(x,t_0)}^2+\sigma^4\abs{v(x,t_0)}^2
+ \sigma^2\abs{\nabla v(x,t_0)}^2}e^{2s\varphi}dx dt\cr
+ \gamma^{-1} \int_Q (\abs{\Delta y}^2 + \abs{\nabla y}^2) e^{2s\varphi}dxdt.
\end{multline}
Now, let $w(x,t) = \text{rot}\; \vv(x,t)$, similarly we apply rot
to the first equation in (\ref{2.10}), we can derive the following equation
\begin{equation}\label{3.3}
 w_{tt}-\mu\Delta w= \text{rot}\; \ff
\qquad \textrm{in}\quad Q.
\end{equation}
Applying  Lemma \ref{L3} of the hyperbolic Carleman estimate, we obtain
\begin{equation}\label{3.4}
C\int_Q  \sigma \Big( \abs{\nabla w}^2
+ \abs{w_t}^2+ \sigma^{2}\abs{w}^2\Big)e^{2s\varphi} dx\,dt
\leq \int_Q \abs{\text{rot}\; \ff}^2e^{2s\varphi}
dx\,dt.
\end{equation}
We recall that, the first equation in (\ref{2.10}), can be written
\begin{equation}\label{2.9}
 \vv_{tt}(x,t)-\mu(x) \Delta \vv = f_0(x,t)
\qquad \textrm{in}\quad Q,
\end{equation}
where,
$$f_0(x,t) =  \ff(x,t)+ (\mu + \lambda) \nabla v + v \nabla \lambda + (\nabla \vv+ (\nabla \vv )^T)\nabla \mu  + \nabla(\lambda^*(x) v_t)- \varrho_1(x) \nabla y. $$
Applying again Lemma \ref{L3} of the hyperbolic Carleman estimate, we obtain
\begin{multline}\label{3.7}
C\int_Q  \sigma \Big( \abs{\nabla \vv}^2
+ \abs{\vv_t(x,t)}^2 + \sigma^{2}\abs{\vv}^2\Big)e^{2s\varphi} dx\,dt
\leq \int_Q \abs{\ff}^2e^{2s\varphi}
dx\,dt\cr
 + \int_Q \Big(\abs{\nabla v_t}^2+ \abs{v_t}^2+  \abs{\nabla v}^2+ \abs{v}^2+ \abs{\nabla y}^2  \Big)e^{2s\varphi}
dx\,dt.
\end{multline}
Collecting (\ref{3.2}), (\ref{3.4}) and (\ref{3.7}), we get
\begin{multline}\label{3.5}
C\int_Q  \Big(\sigma ( \abs{\nabla_{x,t} \vv}^2 + \abs{\nabla_{x,t} w}^2 + \sigma\abs{\nabla v}^2+\sigma^2\abs{v_t}^2 + \abs{\nabla v_t}^2 )
+  \sigma^{3}( \abs{\vv}^2 + \sigma\abs{v}^2 + \abs{w}^2) \Big)e^{2s\varphi} dx\,dt\cr
\leq \int_Q (\abs{\ff}^2 + \abs{\nabla\ff}^2) e^{2s\varphi}
dx\,dt
+  \int_Q ( \gamma^{-1} \abs{\Delta y}^2 + \abs{\nabla y}^2 )e^{2s\varphi}dxdt\cr
+ \int_{Q}\para{ \gamma^{-1}\abs{\Delta v(x,t_0)}^2
+ \gamma^{-1}\abs{v_t(x,t_0)}^2+\sigma^4\abs{v(x,t_0)}^2
+ \sigma^2\abs{\nabla v(x,t_0)}^2}e^{2s\varphi}dx dt.
\end{multline}
We recall that, the second equation in (\ref{2.10}) is given by:
\begin{equation*}
 y_{t}(x,t)- \Delta y (x,t) + \varrho_2(x)  \dive\, \vv_t (x,t) = h (x,t)
\qquad \textrm{in}\quad Q.
\end{equation*}
Furthermore, applying  Lemma \ref{L4.2} of the parabolic Carleman estimate, we obtain
\begin{multline}\label{3}
C\gamma \int_Q \Big(  \abs{\Delta y}^2
+ \sigma^2 \abs{\nabla y}^2
+ \sigma^{4}\abs{y}^2\Big)e^{2s\varphi} dx\,dt\cr
\leq \int_Q \sigma \abs{h}^2e^{2s\varphi}
dx\,dt + \int_Q \sigma\abs{\dive\, \vv_t}^2e^{2s\varphi}
dx\,dt.
\end{multline}
Then we have,
\begin{multline}\label{3.8}
C \gamma^{-1} \int_Q \Big(\abs{\Delta y}^2
+ \sigma^2 \abs{\nabla y}^2
+ \sigma^{4}\abs{y}^2\Big)e^{2s\varphi} dx\,dt\cr
\leq  \gamma^{-2}\int_Q  \sigma \abs{h}^2e^{2s\varphi}
dx\,dt + \gamma^{-2}\int_Q \sigma \abs{v_t}^2e^{2s\varphi}
dx\,dt.
\end{multline}
Inserting (\ref{3.8}) into (\ref{3.5}), we find
\begin{multline}\label{3.50}
C\int_Q  \Big( \sigma\abs{\nabla_{x,t} \vv}^2 + \sigma\abs{\nabla_{x,t} w }^2+ \sigma^2\abs{\nabla v}^2+\sigma^3\abs{v_t}^2 + \sigma\abs{\nabla v_t}^2
+  \sigma^{3} \abs{\vv}^2 + \sigma^3 \abs{w}^2 + \sigma^4\abs{ v}^2
\Big)e^{2s\varphi} dx\,dt\cr
\leq \int_Q (\abs{\ff}^2 + \abs{\nabla \ff}^2 + \gamma^{-2}\sigma \abs{h}^2 ) e^{2s\varphi}
dx\,dt\cr
+ \int_{Q}\para{ \gamma^{-1}\abs{\Delta v(x,t_0)}^2
+ \gamma^{-1}\abs{v_t(x,t_0)}^2+\sigma^4\abs{v(x,t_0)}^2
+ \sigma^2\abs{\nabla v(x,t_0)}^2}e^{2s\varphi}dx dt.
\end{multline}
We deduce that
\begin{multline}\label{4.0}
C\int_Q  \Big(\sigma \abs{\nabla_{x,t} \vv }^2
+  \sigma^{3}\abs{\vv }^2
+ \sigma^4 \abs{ v}^{2} + \sigma^3 \abs{ v_t}^{2} + \sigma^2 \abs{  \nabla v}^{2}  + \sigma \abs{ \nabla v_t}^{2} \Big)e^{2s\varphi} dx\,dt\cr
\leq \int_Q (\abs{\ff}^2 + \abs{\nabla \ff}^2 + \gamma^{-2}\sigma \abs{h}^2 ) e^{2s\varphi}
dx\,dt\cr
+ \int_{Q}\para{ \gamma^{-1}\abs{\Delta v(x,t_0)}^2
+ \gamma^{-1}\abs{v_t(x,t_0)}^2+\sigma^4\abs{v(x,t_0)}^2
+ \sigma^2\abs{\nabla v(x,t_0)}^2}e^{2s\varphi}dx dt.
\end{multline}
Additionally (\ref{3}) and (\ref{4.0}), we find that
\begin{multline}\label{4.1}
C\int_Q  \Big(\sigma \abs{\nabla_{x,t} \vv}^2
+  \sigma^{3}\abs{\vv}^2 + \sigma^4 \abs{ v}^{2}
+ \sigma^3 \abs{ v_t}^{2} + \sigma^2 \abs{  \nabla v}^{2}  + \sigma \abs{ \nabla v_t}^{2}\cr
+ \abs{\Delta y}^2
+ \sigma^2 \abs{\nabla y}^2
+ \sigma^{4}\abs{y}^2\Big)e^{2s\varphi} dx\,dt\cr
\leq \int_Q (\abs{\ff}^2 + \abs{\nabla \ff}^2 + \gamma^{-1}\sigma \abs{h}^2 ) e^{2s\varphi}
dx\,dt\cr
+ \int_{Q}\para{ \gamma^{-1}\abs{\Delta v(x,t_0)}^2
+ \gamma^{-1}\abs{v_t(x,t_0)}^2+\sigma^4\abs{v(x,t_0)}^2
+ \sigma^2\abs{\nabla v(x,t_0)}^2}e^{2s\varphi}dx dt.
\end{multline}
The proof of Theorem \ref{t1} is complete
\section{Proof of the main results}
\setcounter{equation}{0}
This section is devoted to the proof of Theorem \ref{T.1} and Theorem \ref{T.2}. The idea of the proof is based on the Carleman estimate method.\\
A usual methodology by Bellassoued and Yamamoto \cite{[B&Y]}, yields an estimate of a source term $\lambda^*(x)$ and the two spatially varying density.
\subsection{Preliminaries estimate}
\begin{lemma}\label{L2}
Let $\omega$ be an open subdomain of $\Omega$ with regular boundary $\p \omega\supset\Gamma$. There exists constants $\gamma_*$, $s_*$
and $C >0$ such that for any $s \geq s_*$ and any $\gamma \geq \gamma_*$ the following estimate holds:
\begin{equation}
\int_{\omega\times (0, T)} \sigma^4 \abs{v}^2 e^{2 s \varphi} dxdt \leq C \int_{\omega\times (0, T)} \sigma^2 \abs{\nabla v}^2 e^{2 s \varphi} dxdt
\end{equation}
for any $v \in H^1(\omega\times (0, T))$ such that $v(x,t) = 0 $ on $\p\omega \times (0, T)$.
\end{lemma}
\begin{proof}{}
We multiply $\nabla v$ by $(\nabla \varphi) v e^{2 s \varphi}$ and using the divergence theorem, we obtain
\begin{eqnarray*}
\int_\omega \nabla v. (\nabla \varphi) v \; e^{2 s \varphi} dx &=& - \int_\omega v \,\dive ((\nabla \varphi)\, v \;e^{2 s \varphi}) dx\cr
&=& - \int_\omega \abs{v}^2  \Delta\varphi \;e^{2 s \varphi} dx - 2 s \int_\omega \abs{v}^2 \abs{\nabla \varphi}^2 e^{2 s \varphi} dx\cr
&& - \int_\omega \nabla v. (\nabla \varphi) \,v \;e^{2 s \varphi}dx.
\end{eqnarray*}
Therefore,
\begin{eqnarray*}
2\int_\omega  \sigma\nabla v. (\nabla \vartheta)\, v\; e^{2 s \varphi} dx &=& - 2  \int_\omega  \sigma^2 \abs{v}^2 \abs{\nabla \vartheta}^2 e^{2 s \varphi} dx\cr
&& - \int_\omega \sigma \abs{v}^2  \Delta\vartheta\; e^{2 s \varphi} dx  - \gamma \int_\omega \sigma \abs{v}^2  \abs{\nabla\vartheta}^2 e^{2 s \varphi} dx.
\end{eqnarray*}
Taking $\gamma \geq \gamma_*$ and $s \geq s_*$ sufficiently large, we obtain for any $\varepsilon > 0$
\begin{equation}
C \int_\omega \sigma^2 \abs{v}^2 e^{2 s \varphi} dx \leq C_\varepsilon   \int_\omega \abs{\nabla v}^2 e^{2 s \varphi} dx + \varepsilon \int_\omega \sigma^2 \abs{v}^2 e^{2 s \varphi} dx.
\end{equation}
Integrating in $(0, T)$ and taking $\varepsilon$ small we obtain
\begin{equation}
C \int_{\omega\times (0, T)} \sigma^2 \abs{v}^2 e^{2 s \varphi} dx \leq C_\varepsilon   \int_{\omega\times (0, T)} \abs{\nabla v}^2 e^{2 s \varphi} dx.
\end{equation}
Applying the last inequality to $\sigma^2 v$ we obtain
\begin{equation}
 \int_{\omega\times (0, T)} \sigma^4 \abs{v}^2 e^{2 s \varphi} dx \leq C  \int_{\omega\times (0, T)} \sigma^2 \abs{\nabla v}^2 e^{2 s \varphi} dx,
\end{equation}
for any $\gamma\geq \gamma_*$ and $ s\geq  s_*$.
This completes the proof.
\end{proof}
Hence, by Lemma \ref{L2}, we obtain the following Lemma.
\begin{lemma}\label{L2.1}
Let $(\vv, y)\in H^{2}(Q)\times H^{2, 1}(Q)  $, satisfying
\begin{equation}\label{5.7}
\left\{
\begin{array}{lll}
 \vv_{tt}- \Delta_{\mu, \lambda}\vv -  \nabla(\lambda^*(x) \dive \vv_t) + \varrho_1(x)\, \nabla y= \ff  &\,\,
(x,t)\in Q,\cr
 y_t - \Delta y + \varrho_2(x)\, \dive \vv_t = h &\,\,
(x,t) \in Q,\cr
\vv = 0,\qquad y =0 &\,\,
(x,t) \in \Sigma,
\end{array}
\right.
\end{equation}
and
\begin{equation}\label{5.9}
\begin{array}{ll}
\p_t^j\vv(x,0)=\p^j_t\vv(x,T)=0,\quad y(x,0)=y(x,T)=0 & \quad\textrm{for all}\quad x\in\Omega,\, j=0,1.
\end{array}
\end{equation}
There exist positive constants $\gamma_*$ and $C>0$ such that, for any $\gamma \geq \gamma_*$ we can find $s_*$ and $D$, the following inequality holds
\begin{multline}\label{7.2}
\int_Q e^{2 s \varphi}\Big( \sigma \abs{\nabla_{x,t} \vv}^2 + \sigma^3 \abs{\vv}^2
+ \sigma^4 \abs{ \dive \vv}^{2} + \sigma^3 \abs{ \dive \vv_t}^{2} + \sigma^2 \abs{  \nabla\dive \vv}^{2}  + \sigma \abs{ \nabla\dive \vv_t}^{2}\cr
+ \abs{\Delta y}^2 + \sigma^2 \abs{\nabla y}^2 + \sigma^4 \abs{y}^2 \Big) dxdt
\leq C \int_Q \Big( \gamma^{-1} \sigma \abs{h}^2 +  \sigma^2 \abs{\ff}^2 +  \abs{\nabla\ff}^2 \Big) dxdt\cr
+ C e^{D s}( \norm{\vv}^2_{H^4(\omega\times (0,T)}+\norm{\dive\vv(\cdot,t_0)}^2_{H^2(\Omega)}
+\norm{\dive\vv_t(\cdot,t_0)}^2_{L^2(\Omega)})
\end{multline}
for any $s \geq s_*$.
\end{lemma}
\begin{proof}
Let $\omega' \subset \omega$ such that $\p\omega'\supset\Gamma$. In order to apply Carleman estimate, we introduce a cut-off function $\chi$ satisfying $0 \leq \chi \leq 1$, $\chi \in \mathscr{C}^\infty(\R^3)$, $\chi = 1$ in $\overline{\Omega\backslash \omega'}$ and supp$\chi \subset \Omega$.\\
Put
$$\widehat{\vv} (x,t)= \chi(x) \vv(x,t), \qquad  \widehat{y} (x,t)= \chi(x) y(x,t),$$
Noting that $(\widehat{\vv}, \widehat{y})\in H^{2}(Q)\times H^{2, 1}(Q)  $, satisfying
\begin{equation}\label{5.8}
\left\{
\begin{array}{lll}
 \widehat{\vv}_{tt} - \Delta_{\mu, \lambda}\widehat{\vv} -  \nabla(\lambda^*\dive \widehat{\vv}_t) + \varrho_1 \nabla \widehat{y}=\widehat{\ff}  &\,\,
(x,t) \in Q,\cr
 \widehat{y}_t - \Delta \widehat{y} + \varrho_2 \dive \widehat{\vv}_t =\widehat{g}  &\,\,
(x,t)\in Q,
\end{array}
\right.
\end{equation}
with
$$
\textrm{Supp}(\widehat{\vv}(\cdot,t))\subset\Omega\quad \textrm{Supp}(\widehat{y}(\cdot,t))\subset\Omega, \quad \forall\,t\in(0,T),
$$
where
$$
\widehat{\ff}= \chi(x) \ff(x,t)
 - [\Delta_{\mu, \lambda}, \chi] \vv - \nabla(\lambda^* \nabla \chi. \vv_t) - \lambda^* (\nabla\chi) \dive \vv_t + \varrho_1 y \nabla\chi,
$$
$$
\widehat{g}=\chi(x) h(x,t) - 2 \nabla\chi \nabla y - y \Delta\chi + \varrho_2 \nabla\chi\cdot \vv_t
$$
and
$$[\Delta_{\mu, \lambda}, \chi] \vv = \Delta_{\mu, \lambda}( \chi \vv) - \chi \Delta_{\mu, \lambda} \vv .
$$
 Noting that $(\widehat{\vv},\widehat{y})$ satisfies (\ref{5.8}), then we can apply the Carleman estimate  for Biot's system (\ref{11}) to $(\widehat{\vv}, \widehat{y})$, we obtain
\begin{multline}
C \int_{(\Omega\setminus \omega') \times (0, T)} e^{2 s \varphi}\Big( \sigma \abs{\nabla_{x,t} \vv}^2 + \sigma^3 \abs{\vv}^2
+ \sigma^4 \abs{ \dive \vv}^{2} + \sigma^3 \abs{ \dive \vv_t}^{2} + \sigma^2 \abs{  \nabla\dive \vv}^{2}  + \sigma \abs{ \nabla\dive \vv_t}^{2}\cr
+ \abs{\Delta y}^2 + \sigma^2 \abs{\nabla y}^2 + \sigma^4 \abs{y}^2 \Big) dxdt
\leq  \int_Q e^{2 s \varphi} \Big(  \gamma^{-1} \sigma \abs{\widehat{g}}^2+\abs{\widehat{\ff}}^2+\abs{\nabla\widehat{\ff}}^2\Big)dxdt\cr
+ \int_{Q}\para{ \gamma^{-1}\abs{\Delta \dive \widehat{\vv} (x,t_0)}^2
+ \gamma^{-1}\abs{\dive \widehat{\vv}_t(x,t_0)}^2+\sigma^4\abs{\dive \widehat{\vv}(x,t_0)}^2
+ \sigma^2\abs{\nabla \dive \widehat{\vv}(x,t_0)}^2}e^{2s\varphi}dx dt.
\end{multline}
By a simple computation, we get
\begin{multline}
 \int_Q e^{2 s \varphi} \Big(  \gamma^{-1} \sigma \abs{\widehat{g}}^2+\abs{\widehat{\ff}}^2+\abs{\widehat{\nabla\ff}}^2\Big) dxdt
\leq C \int_Q \Big( \gamma^{-1} \sigma \abs{ \chi h}^2 +  \abs{\chi \ff}^2 +  \abs{\nabla(\chi \ff)}^2 \Big) dxdt\cr
+ C  \int_Q \Big( \abs{ (\Delta\chi)  \vv}^2 +   \abs{ \nabla \chi}^2  \abs{ \nabla \vv}^2 +  \abs{\nabla [\Delta_{\mu, \lambda}, \chi] \vv}^2 \Big)e^{2 s \varphi} dxdt\cr
+ C  \int_Q \Big(    \abs{\nabla(\lambda^* \nabla\chi.  \vv_t)}^2+ \abs{  \lambda^* (\nabla\chi) \dive \vv_t}^2 +  \abs{\nabla(\nabla(\lambda^* \nabla\chi. \vv_t))}^2+ \abs{ \nabla(\lambda^* (\nabla\chi) \dive \vv_t)}^2  \Big)e^{2 s \varphi} dxdt\cr
+ C   \int_Q \Big(  \abs{(\Delta\chi) y}^2 +  \abs{\nabla \chi}^2  \abs{\nabla y}^2 +  \abs{ \nabla \chi}^2  \abs{y}^2 \Big)e^{2 s \varphi} dxdt.
\end{multline}
Since Supp$\Delta \chi$, Supp$\nabla \chi \subset \omega'$ we obtain
\begin{multline}\label{7.3}
C \int_{(\Omega\setminus \omega' )\times (0, T)} e^{2 s \varphi}\Big( \sigma \abs{\nabla_{x,t} \vv}^2 + \sigma^3 \abs{\vv}^2
+ \sigma^4 \abs{ \dive \vv}^{2} + \sigma^3 \abs{ \dive \vv_t}^{2} + \sigma^2 \abs{  \nabla\dive \vv}^{2}  + \sigma \abs{ \nabla\dive \vv_t}^{2}\cr
+ \abs{\Delta y}^2 + \sigma^2 \abs{\nabla y}^2 + \sigma^4 \abs{y}^2 \Big) dxdt
\leq C \int_Q \Big( \gamma^{-1} \sigma \abs{  h}^2 +  \abs{\ff}^2 +  \abs{\nabla \ff}^2 \Big) dxdt
+ C e^{D s} \norm{\vv}^2_{H^4(\omega \times (0, T))} \cr
+ C \Big( \int_{\omega'\times (0, T)} \sigma^4 \abs{y}^2  e^{2 s \varphi} dxdt + \int_{\omega'\times (0, T)} \sigma^2 \abs{\nabla y}^2  e^{2 s \varphi} dxdt\Big)\cr 
+ \int_{Q}\para{ \gamma^{-1}\abs{\Delta \dive \vv(x,t_0)}^2
+ \gamma^{-1}\abs{\dive \vv_t(x,t_0)}^2+\sigma^4\abs{\dive \vv(x,t_0)}^2
+ \sigma^2\abs{\nabla \dive \vv(x,t_0)}^2}e^{2s\varphi}dx dt.
\end{multline}
Let $\chi_1$ be a cut-off function satisfying $0 \leq \chi_1 \leq 1$, $\chi_1\in C^{\infty} (\R^3)$, $\chi_1 = 1$ in $\overline{\omega'}$
and Supp$(\chi_1)\subset \omega$. Let us consider $z(x,t) = \chi_1(x) y(x, t)\in H^1(\omega\times (0, T)) $ and $z(x, t) = 0$ for all $(x, t)\in \p \omega \times (0, T) $, so that by Lemma \ref{L2}, we have
\begin{multline}
\int_{\omega'\times (0, T)} \sigma^4 \abs{y}^2 \sigma e^{2 s \varphi} dxdt \leq \int_{\omega\times (0, T)} \sigma^4 \abs{z}^2 e^{2 s \varphi} dxdt\cr
\leq C \int_{\omega\times (0, T)} \sigma^2 \abs{\nabla y}^2 e^{2 s \varphi} dxdt + C \int_Q \sigma^2 \abs{y}^2 e^{2 s \varphi} dxdt.
\end{multline}
Furthermore by the first equation of (\ref{5.7}), we have
\begin{equation}\label{7.4}
\int_{\omega'\times (0, T)} \sigma^2 \abs{\nabla y}^2  e^{2 s \varphi} dxdt \leq C e^{D s} \norm{\vv}^2_{H^3(\omega \times (0, T))} + \int_{\Omega\times (0, T)}\sigma^2 \abs{\ff}^2 e^{2 s \varphi} dxdt.
\end{equation}
Inserting (\ref{7.4}) 
in (\ref{7.3}), we obtain (\ref{7.2}).
This completes the proof of the lemma.
\end{proof}

Henceforth we fix $\gamma>0$ sufficiently large. By $\n$ we denote the quantity
\begin{multline}\label{5.4}
\n (\vv, y) =  \int_Q e^{2 s \varphi}\Big( s \abs{\nabla_{x,t} \vv}^2 + s^3 \abs{\vv}^2
+ s^4 \abs{ \dive \vv}^{2} + s^3 \abs{ \dive \vv_t}^{2}\cr
+ s^2 \abs{  \nabla\dive \vv}^{2}  + s \abs{ \nabla\dive \vv_t}^{2}
+ \abs{\Delta y}^2 + s^2 \abs{\nabla y}^2 + s^4 \abs{y}^2 \Big) dxdt.
\end{multline}
We introduce a cut-off function $\eta$ satisfying $0\leq \eta\leq 1$, $\eta\in \mathcal{C}^{\infty} (\R)$, $\eta = 1$ in $(2\varepsilon , T-2\varepsilon)$  and\\
Supp$(\eta)\subset (\varepsilon, T-\varepsilon)$. Finally we denote
$$
 \widetilde{\vv} = \eta \vv,\qquad \widetilde{y}= \eta y.
 $$
Setting $d_0=e^{(\beta M-\delta/2)\gamma}$, we have
$$
\max_{x\in\Omega}\psi(x,t)\leq d_0,\quad t\in (0,2\varepsilon)\cup (T-2\varepsilon,T).
$$

\begin{lemma}\label{L3.3}
There exist three
positive constant $s_*$, $C>0$ and $D$ such that
the following inequality holds:
\begin{multline*}
C\n(\widetilde{\vv},\widetilde{y})\leq \int_Q \para{s\abs{h}^2+s^2\abs{\ff}^2+\abs{\nabla \ff}^2}e^{2s\varphi}dxdt\cr
+ e^{Ds}\para{\norm{\vv}^2_{H^4(\omega\times(0,T))}+\norm{\dive\vv(\cdot,t_0)}_{H^2(\Omega)}^2+
\norm{\dive\vv_t(\cdot,t_0)}_{L^2(\Omega)}^2 }\cr
+ Cs^2e^{2d_0s}\para{\norm{\vv}
^2_{H^1(0,T;H^1(\Omega))} + \norm{y}_{L^2(Q)}^2}
\end{multline*}
for any $s\geq s_*$ and any $(\vv,y)\in H^2(Q)\times H^{2,1}(Q)$ satisfying
$$
\begin{array}{lll}
\vv_{tt}-\Delta_{\mu,\lambda} \vv-\nabla(\lambda^*\dive\vv_t)+\varrho_1\nabla y=\ff & (x,t)\in Q,\cr
y_t-\Delta y +\varrho_2\,\dive\, \vv_t  = h & (x,t)\in Q,\cr
\vv=0,\,\,y=0 & (x,t)\in \Sigma.
\end{array}
$$
\end{lemma}
\begin{proof}
Let $(\vv,y)\in H^2(Q)\times H^{2,1}(Q)$. Put
$$
\widetilde{\vv}(x,t)=\eta(t)\vv(x,t),\qquad \widetilde{y}(x,t)=\eta(t)y(x,t).
$$
Noting that $(\widetilde{\vv},\widetilde{y})
\in H^2(Q)\times H^{2,1}(Q)$ satisfies
\begin{equation}\label{3.10}
\begin{array}{lll}
\widetilde{\vv}_{tt}-\Delta_{\mu,\lambda} \widetilde{\vv}-\nabla(\lambda^*\dive\vv_t)
+ \varrho_1\nabla \widetilde{y}
= \eta\ff+\eta_{tt}\vv+2\eta_{t}\vv_{t}
& (x,t)\in Q,\cr
\widetilde{y}_t-\Delta \widetilde{y}
+\varrho_2\,\dive\, \widetilde{\vv}_t
= \eta h+\eta_t(y-\varrho\dive\vv) & (x,t)\in Q,
\cr
\widetilde{\vv}=0,\,\,\widetilde{y}=0 & (x,t)\in \Sigma,
\end{array}
\end{equation}
and applying Carleman estimate (\ref{7.2})
to $(\widetilde{\vv},\widetilde{y})$, we obtain
\begin{multline*}
C\n(\widetilde{\vv},\widetilde{y})\leq \int_Q\para{s\abs{h}^2
+ s^2\abs{\ff}^2 + \abs{\nabla\ff}^2}
e^{2s\varphi}dxdt\cr
+ e^{Ds}\para{\norm{\vv}^2
_{H^4(\omega \times (0, T))}+\norm{\dive\vv(\cdot,t_0)}_{H^2(\Omega)}^2+
\norm{\dive\vv_t(\cdot,t_0)}_{L^2(\Omega)}^2 }\cr
+\int_Q
(s^2(\vert\eta_{tt}\vert^2+\vert\eta_t\vert^2)
(\vert y\vert^2 + \vert \vv\vert^2
+ \vert \vv_t\vert^2 + \vert \nabla \vv\vert^2
+ \vert \nabla\vv_t\vert^2) e^{2s\varphi}dxdt\cr
\end{multline*}
for any $\gamma\geq\gamma_*$ and $s\geq s_*$. Since $\textrm{Supp}(\eta_{tt}),\,\textrm{Supp}(\eta_t)\subset (0,2\epsilon)\cup(T-2\epsilon,T)$, we obtain from (\ref{2.6})
\begin{multline*}
\int_Q
(s^2(\vert\eta_{tt}\vert^2+\vert\eta_{t}\vert^2)
(\vert y\vert^2 + \vert \vv\vert^2
+ \vert \vv_t\vert^2 + \vert \nabla\vv\vert^2
+ \vert \nabla\vv_t\vert^2) e^{2s\varphi}dxdt\\
\leq Cs^2e^{2d_0s}\para{\norm{\vv}
_{H^1(0,T;H^1(\Omega))}^2+\norm{y}^2_{L^2(Q)}}.
\end{multline*}
This completes the proof of the lemma.
\end{proof}
\begin{lemma}\label{L4}
There exists a positive constant $C>0$ such that the following estimate
$$\int_\Omega \abs{z(x, t_0)}^2 dx \leq C \int_Q (\sigma \abs{z(x,t)}^2 + \sigma^{-1} \abs{z_t(x,t)}^2) dxdt $$
for any $z\in H^1(0, T; L^2(\Omega))$.
\end{lemma}
\begin{proof}
By direct computations, we have
\begin{eqnarray*}
\int_\Omega \eta^2(t_0) \abs{z(x,t_0)}^2 dx &=& \int_{0}^{t_0} \frac{d}{dt} \Big(\int_\Omega \eta^2(t) \abs{z(x, t)}^2  dx\Big) dt \cr
&=& 2 \int_{0}^{t_0}\int_\Omega \eta^2(t)z(x, t) z_t(x, t)dxdt + 2 \int_{0}^{t_0}\int_\Omega \eta_t(t) \eta(t)  \abs{z(x, t)}^2 dxdt.
\end{eqnarray*}
Then, we have
\begin{equation*}
\int_\Omega \abs{z(x, t_0)}^2 dx \leq C \int_Q (\sigma \abs{z(x,t)}^2 + \sigma^{-1} \abs{z_t(x,t)}^2) dxdt.
\end{equation*}
This complete the proof of the lemma.
\end{proof}


Finally, let $\varphi(x,t)$ be the weight function defined by
$$ \varphi(x,t)= e^{\gamma\psi(x,t)} := \rho(x) \alpha(t), $$
where, $\rho(x)$ and $\alpha(t)$ are defined by
\begin{equation}
\rho(x) = e^{\gamma(\abs{x - x_0}^2 + \beta M)}\geq e^{\gamma\beta M} \equiv d, \forall x\in \Omega\quad \text{and}\qquad
\alpha(t) = e^{- \beta \gamma (t - t_0)^2} \leq 1, \forall t\in (0, T).
\end{equation}
\subsection{Proof of the stability in determing $\lambda^*$ }
We prove now Theorem \ref{T.1}. To this end we use the global Carleman estimate (\ref{11}).\\

Consider now the following system
\begin{equation}\label{5.1}
\left\{
\begin{array}{lll}
 \uu_{tt} - \Delta_{\mu, \lambda}\uu -  \nabla(\lambda^*(x) \dive \uu_t) + \varrho_1(x) \nabla \theta = 0  &\,\,
(x,t) \in Q,\cr
 \theta_t - \Delta \theta + \varrho_2(x) \dive \uu_t = 0 &\,\,
(x,t) \in Q,\cr
\uu = 0,\qquad \theta =0 &\,\,
(x,t) \in \Sigma.
\end{array}
\right.
\end{equation}
Henceforth, for simplicity, consider the following functions:
$$
\uu =\uu(\lambda^*, \varrho_1, \varrho_2),\quad \widetilde{\uu} = \uu (\widetilde{\lambda}^*, \varrho_1, \varrho_2),\qquad \theta = \theta(\lambda^*, \varrho_1, \varrho_2),\quad \widetilde{\theta} = \theta(\widetilde{\lambda}^*, \varrho_1, \varrho_2)
$$
and
$$
\vv=\uu-\widetilde{\uu},\quad y=\theta -\widetilde{\theta},\quad f=\lambda^*-\widetilde{\lambda}^*.
$$
Then by (\ref{5.1}), we easily see that
\begin{equation}\label{6.1}
\left\{
\begin{array}{lll}
 \vv_{tt}- \Delta_{\mu, \lambda}\vv -  \nabla(\lambda^* \dive \vv_t) + \varrho_1 \nabla y =\nabla(f(x)\dive\widetilde{\uu}_t)  &\,\,
(x, t) \in Q,\cr
 y_t - \Delta y + \varrho_2 \dive \vv_t=0  &\,\,
(x, t) \in Q,\cr
\vv = 0,\qquad y =0 &\,\,
(x, t) \in \Sigma.
\end{array}
\right.
\end{equation}
In this subsection we discuss a linearized inverse problem of determining $\lambda^*$.
We assume that the assumptions $(A.1)$ and $(A.2)$ holds true, then our inverse problem is identification of $f(x)$. \\
Let
$$
\vv^i=\p_t^i \vv,\quad y^i= \p_t^i y,\quad i=1,2,3.
$$
 By a simple calculation, we obtain for $i=1,2,3$,
$(\vv^i, y^i)$ satisfies
\begin{equation}\label{6.2}
\left\{
\begin{array}{lll}
 \vv^i_{tt}- \Delta_{\mu, \lambda}\vv^i -  \nabla(\lambda^*(x) \dive \vv^i_t) + \varrho_1(x) \nabla y^i =\nabla(f(x)\dive\widetilde{\uu}^i_t)  & \,\,
(x, t) \in Q,\cr
 y^i_t - \Delta y^i + \varrho_2(x) \dive \vv^i_t =0  & \,\,
(x, t) \in Q,\cr
\vv^i = 0,\qquad y^i =0 &\,\,
(x, t) \in \Sigma.
\end{array}
\right.
\end{equation}
Let $v=\dive\vv$,  we apply $\dive$ to the first equation of the system (\ref{6.1}), we get
\begin{multline}
\Delta (f(x)\dive\widetilde{\uu}_t) (x,t_0)=\cr
(v_{tt}-(2\mu+\lambda)\Delta v -\Delta(\lambda^*(x) v_t)+\varrho_1(x)\Delta y + \nabla\varrho_1(x). \nabla y )(x,t_0)\equiv k(x,t_0).
\end{multline}
We have $f(x)=0$ and $\nabla f(x)=0$ on $\Gamma$ and
$$
\abs{\dive\widetilde{\uu}_t (x,t_0)}\geq\varepsilon,
$$
then by the elliptic Carleman estimate, we deduce that
\begin{multline*}
\int_\Omega\para{s^3\abs{f(x)}^2+s\abs{\nabla f(x)}^2+s^{-1}\abs{\Delta f(x)}^2}e^{2s\varphi(x,t_0)}dx\leq C\int_\Omega\abs{k(x,t_0)}^2e^{2s\varphi(x,t_0)}dx\cr
\leq C\int_\Omega\abs{v_{tt}(x,t_0)}^2e^{2s\varphi(x,t_0)}dx\cr
+e^{Ds}\para{\norm{\dive\vv(\cdot,t_0)}^2_{H^2(\Omega)}+
\norm{\dive\vv_t(\cdot,t_0)}^2_{H^2(\Omega)}+\norm{y(\cdot,t_0)}^2_{H^2(\Omega)}
}.
\end{multline*}
Applying lemma \ref{L4} to  $z(x,t) = \eta(t) e^{2 s \varphi(x, t)} v_{tt}(x,t) = e^{2 s \varphi(x, t)} \widetilde{v}_{tt}(x,t) $, we get
\begin{multline*}
\int_\Omega\abs{v_{tt}(x,t_0)}^2e^{2s\varphi(x,t_0)}dx
\leq C\int_Q\para{s\abs{\widetilde{v}_{tt}(x,t)}^2+s^{-1}\abs{\widetilde{v}_{ttt}(x,t)}^2}e^{2s\varphi(x,t)}dxdt\cr
\leq C\int_Q\para{s\abs{\dive\widetilde{\vv}^1_{t}(x,t)}^2+s^{-1}\abs{\dive\widetilde{\vv}^2_{t}(x,t)}^2}e^{2s\varphi(x,t)}dxdt\cr
\leq C\para{s^{-2}\n(\widetilde{\vv}^1,\widetilde{y}^1)+s^{-4}\n(\widetilde{\vv}^2,\widetilde{y}^2)}.
\end{multline*}
We deduce that
\begin{multline}\label{7.6}
\int_\Omega\para{s^5\abs{f(x)}^2+s^3\abs{\nabla f(x)}^2+s\abs{\Delta f(x)}^2}e^{2s\varphi(x,t_0)}dx\cr
\leq C\para{\n(\widetilde{\vv}^1,\widetilde{y}^1)+s^{-2}\n(\widetilde{\vv}^2,\widetilde{y}^2)}\cr
+ s^2 e^{Ds}\para{\norm{\dive\vv(\cdot,t_0)}^2_{H^2(\Omega)}+
\norm{\dive\vv_t(\cdot,t_0)}^2_{H^2(\Omega)}+\norm{y(\cdot,t_0)}^2_{H^2(\Omega)}
}.
\end{multline}
Now, applying lemma \ref{L3.3} to $(\widetilde{\vv}^i, \widetilde{y}^i)$, for $i=1,2$, we obtain
\begin{multline}\label{7.7}
C\n(\widetilde{\vv}^i,\widetilde{y}^i)\leq \int_Q \para{s^2\abs{\nabla(f(x)\dive\widetilde{\uu}^i_t)}^2+\abs{\Delta(f(x)\dive\widetilde{\uu}^i_t)}^2}e^{2s\varphi}dxdt\cr
+ e^{Ds}\para{\norm{\vv^i}^2_{H^4(\omega\times (0, T))}+\norm{\dive\vv^i(\cdot,t_0)}_{H^2(\Omega)}^2+
\norm{\dive\vv^i_t(\cdot,t_0)}_{L^2(\Omega)}^2 }\cr
+ Cs^2e^{2d_0s}\para{\norm{\vv^i}
^2_{H^1(0,T;H^1(\Omega))} + \norm{y^i}_{L^2(Q)}^2}.
\end{multline}
Inserting (\ref{7.7}) into (\ref{7.6}), and using Assumption $(A.2)$, we obtain
\begin{multline*}
\int_\Omega \para{s^5\abs{f(x)}^2+s^3\abs{\nabla f(x)}^2+s\abs{\Delta f(x)}^2}e^{2s\varphi(x,t_0)}dx \cr
\leq e^{Ds}\sum_{i=1,2} \para{\norm{\vv^i}^2_{H^4(\omega\times (0, T))}+\norm{\dive\vv^i(\cdot,t_0)}_{H^2(\Omega)}^2+
\norm{\dive\vv^i_t(\cdot,t_0)}_{H^2(\Omega)}^2+\norm{y(\cdot,t_0)}^2_{H^2(\Omega)}}
+ Cs^2e^{2d_0s}M_0
\end{multline*}
We deduce that
\begin{multline*}
\int_\Omega \para{s^5\abs{f(x)}^2+s^3\abs{\nabla f(x)}^2+s\abs{\Delta f(x)}^2}e^{2s\varphi(x,t_0)}dx \cr
\leq e^{Ds} \para{\norm{\vv}^2_{H^6(\omega\times (0, T))}+\norm{y(\cdot,t_0)}^2_{H^2(\Omega)}+
\sum_{i=1}^3\norm{\dive\p_t^i\vv(\cdot,t_0)}_{H^2(\Omega)}^2}
+ Cs^2e^{2d_0s}M_0\cr
\leq e^{Ds} \para{\norm{\vv}^2_{H^6(\omega\times (0, T))}+N_{t_0}(\vv,y)}
+ Cs^2e^{2d_0s}M_0
\end{multline*}

%
Finally, minimizing the right hand side with respect to $s$, we obtain:  there exist $\kappa\in (0, 1)$ such that 
\begin{equation}
\norm{f}^2_{H^2(\Omega)}
\leq C \para{\norm{\vv}_{H^6(\omega\times(0,T)}^2+ N_{t_0}(\vv, y)}^\kappa.
\end{equation}
The proof of Theorem \ref{T.1} is completed.

\subsection{Proof of the stability estimate in determining $\varrho_1$ and $\varrho_2$}
For simplicity, we set
$$\uu = \uu(\lambda^*, \varrho_1, \varrho_2), \quad \uu^* = \uu(\lambda^*, \widetilde{\varrho}_1, \widetilde{\varrho}_2) $$
and
$$\theta = \theta(\lambda^*, \varrho_1, \varrho_2), \quad \theta^* = \theta(\lambda^*, \widetilde{\varrho}_1, \widetilde{\varrho}_2). $$
Let $(\uu, \theta)$ satisfies the following equation

\begin{equation}\label{8.1}
\left\{
\begin{array}{lll}
 \uu_{tt}(x,t)- \Delta_{\mu, \lambda}\uu(x,t) -  \nabla(\lambda^* \dive \uu_t(x,t)) + \varrho_1 \nabla \theta (x,t)=0  & \textrm{in }\,\,
Q\cr
 \theta_t(x,t) - \Delta \theta(x,t) + \varrho_2 \dive \uu_t(x,t)=0  & \textrm{in }\,\,
Q\cr
\uu(x,t) = 0,\qquad \theta(x,t) =0 &\textrm{on }\,\,
\Sigma
\end{array}
\right.
\end{equation}

and $(\uu^*, \theta^*)$ satisfies the following equation

\begin{equation}\label{8.2}
\left\{
\begin{array}{lll}
 \uu^*_{tt}(x,t)- \Delta_{\mu, \lambda}\uu^*(x,t) -  \nabla(\lambda^* \dive\uu^*_t(x,t)) + \varrho_1 \nabla\theta^* (x,t)=0  & \textrm{in }\,\,
Q\cr
 \theta_t^*(x,t) - \Delta\theta^*(x,t) + \varrho_2 \dive \uu_t^* (x,t)=0  & \textrm{in }\,\,
Q\cr
\uu^*(x,t) = 0,\qquad \theta^*(x,t) =0 &\textrm{on }\,\,
\Sigma
\end{array}
\right.
\end{equation}

Let $$\vv = \uu - \uu^*, \qquad  y = \theta - \theta^*$$

$$p(x)=\widetilde{\varrho}_1(x) - \varrho_1 (x), \quad \text{and}\quad q(x)=\widetilde{\varrho}_2(x) -  \varrho_2 (x), $$

where $(\uu, \theta)$ satisfies (\ref{8.1})
and $(\uu_*, \theta_*)$ satisfies (\ref{8.2}).
Then, by a simple calculation, we have  

\begin{equation}\label{8.3}
\left\{
\begin{array}{lll}
 \vv_{tt}- \Delta_{\mu, \lambda}\vv -  \nabla(\lambda^* \dive \vv_t) + \varrho_1 \nabla y = p(x) \nabla \theta^*  & \,\,
(x,t)\in Q,\cr
 y_t - \Delta y + \varrho_2 \dive \vv_t = q(x) \dive \uu_t^* &\,\,
(x,t) \in Q,\cr
\vv = 0,\qquad y =0 &\,\,
(x,t)\in\Sigma.
\end{array}
\right.
\end{equation}
In this subsection we discuss a linearized inverse problem of determining $\varrho_1$ and $\varrho_2$.
We assume that the assumptions $(A.1)$ and $(A.2)$ holds true, then our inverse problem is identification of $p(x)$ and $q(x)$. \\

For $i=1,2,3$, we denote $\vv^i = \p_t^i \vv$, $y^i = \p_t^i y$, where $(\vv, y)$  satisfies (\ref{8.3}).
Then,  we have
\begin{equation}\label{8.4}
\left\{
\begin{array}{lll}
 \vv^i_{tt}- \Delta_{\mu, \lambda}\vv^i -  \nabla(\lambda^* \dive \vv^i_t) + \varrho_1 \nabla y^i = p(x) \nabla \p_t^i\theta^*  & \,\,
(x,t)\in Q,\cr
 y^i_t - \Delta y^i + \varrho_2 \dive \vv^i_t= q(x) \dive \p_t^i\uu_t^* &\,\,
(x,t) \in Q,\cr
\vv^i = 0,\qquad y^i =0 &\,\,
(x,t)\in\Sigma.
\end{array}
\right.
\end{equation}
 We apply lemma \ref{L3.3} to $(\widetilde{\vv^i}, \widetilde{y^i})$, for $i=1,2,3$, we obtain
\begin{multline}\label{8.5}
C\n(\widetilde{\vv}^i,\widetilde{y}^i)\leq \int_Q \para{s \abs{q(x) \dive \p_t^i\uu^*_t}^2 +  s^2\abs{p(x) \nabla\p_t\theta^* }^2+\abs{\nabla(p(x) \nabla\p_t^i \theta^*) }^2}e^{2s\varphi}dxdt\cr
+ e^{Ds}\para{\norm{\vv^i}^2_{H^4(\omega\times (0, T))}+\norm{\dive\vv^i(\cdot,t_0)}_{H^2(\Omega)}^2+
\norm{\dive\vv^i_t(\cdot,t_0)}_{L^2(\Omega)}^2 }\cr
+ Cs^2e^{2d_0s}\para{\norm{\vv^i}
^2_{H^1(0,T;H^1(\Omega))} + \norm{y^i}_{L^2(Q)}^2},
\end{multline}
where $$ \widetilde{\vv}^i = \eta (t) \vv^i,\qquad \widetilde{y}^i = \eta (t) y^i,\qquad i=1, 2, 3.$$
In terms of Assumption $(A.2)$, we obtain
\begin{multline}
C\n(\widetilde{\vv}^i,\widetilde{y}^i)
\leq \int_Q \para{s \abs{q(x)}^2 +  s^2\abs{p(x)}^2+\abs{\nabla p(x)}^2}e^{2s\varphi}dxdt\cr
+ e^{Ds}\para{\norm{\vv^i}^2_{H^4(\omega\times (0, T))}+\norm{\dive\vv^i(\cdot,t_0)}_{H^2(\Omega)}^2+
\norm{\dive\vv^i_t(\cdot,t_0)}_{L^2(\Omega)}^2 }
+ Cs^2e^{2d_0s}M_0.
\end{multline}

In the following we give two lemmas which will be used in the proof of the stability theorem of determining of the two spatially varying coefficients.
The first one is a Carleman estimate for the first-order partial differential equation:

We consider a first order partial differential equation:
\begin{equation}\label{r1}
P(x, D)\ff = \displaystyle\sum_{j=1}^{3} \gamma_j(x) \p_j \ff + \gamma_0(x) \ff, \qquad x\in \Omega,
\end{equation}
where
\begin{equation}\label{r2}
 \gamma_0(x)\in W^{1, \infty}(\Omega),\qquad \gamma(x)=(\gamma_1(x), \gamma_2(x), \gamma_3(x))\in (W^{1, \infty}(\Omega))^3
\end{equation}
and
\begin{equation}\label{r3}
\abs{\gamma(x). (x- x_0)}\geq c_0>0,\qquad \text{on}\; \overline{\Omega},
\end{equation}
with a constant $c_0>0$. Then we have the following Lemma
\begin{lemma}\label{L7}
In addition to (\ref{r2}) and (\ref{r3}). Then, there exist constants $s_*>0$ and $C>0$ such that
\begin{equation}\label{6.6}
s^2 \int_\Omega  e^{2s \rho (x)} \abs{\ff}^2 dx \leq  C \int_\Omega e^{2 s \varphi(x, t)}\abs{P(x,D) \ff}^2 dx dt
\end{equation}
 and
\begin{equation}\label{6.7}
s^2 \int_\Omega  e^{2s \rho (x)} \abs{\nabla \ff}^2 dx \leq
 C \int_\Omega e^{2 s \varphi(x, t)}(\abs{P(x, D)  \ff}^2 + \abs{\nabla (P(x,D) \ff)}^2 ) dx dt,
\end{equation}
 for all $\ff\in H^2(\Omega)$ satisfying $\ff(x) = 0$, $\nabla \ff(x) = 0$, $x\in \Gamma$ and all  $s> s_*$.

\end{lemma}
\begin{proof}
Let denote $\zeta(x)= \gamma(x)\cdot(x- x_0) $.
We multiply the both sides of (\ref{r1}) by $ \ff(x) \zeta(x) e^{2 s \rho (x)}$ and using the divergence theorem, we obtain
\begin{multline}\label{r4}
\int_\Omega  P(x, D) \ff(x) \ff(x) \zeta(x) e^{2 s \rho (x)} dx\cr
= \int_\Omega \nabla\ff(x).\gamma(x) \ff(x) \zeta(x) e^{2 s \rho (x)} dx
+ \int_\Omega\gamma_0(x) \abs{\ff(x)}^2  \zeta(x) e^{2 s \rho (x)} dx\cr
= - \int_\Omega \ff(x) \dive( \gamma(x)\ff(x) \zeta(x) e^{2 s \rho (x)}) dx
+ \int_\Omega\gamma_0(x) \abs{\ff(x)}^2  \zeta(x) e^{2 s \rho (x)} dx\cr
= - \int_\Omega \abs{\ff(x)}^2  \dive(\zeta(x) \gamma(x) ) e^{2 s \rho (x)} dx
- \int_\Omega \gamma(x)\cdot\nabla\ff(x) \zeta(x) \ff(x)  e^{2 s \rho (x)} dx\cr
- 2 s \int_\Omega \abs{\ff(x)}^2   \zeta^2(x) e^{2 s \rho (x)} dx
+ \int_\Omega\gamma_0(x) \abs{\ff(x)}^2  \zeta(x) e^{2 s \rho (x)} dx.
\end{multline}
Using (\ref{r3}), and the fact that
$$ \gamma(x) \cdot \nabla\ff(x) = P(x, D) \ff - \gamma_0(x) \ff(x)$$
We deduce, in terms of (\ref{r4}) and the Cauchy-Schwartz inequality,
\begin{eqnarray*}
2 c_0^2 s \int_\Omega \abs{\ff(x)}^2  e^{2 s \rho (x)} dx &\leq & 2 s \int_\Omega \abs{\ff(x)}^2   (\zeta(x))^2 e^{2 s \rho (x)} dx \cr
&\leq& 2 \int_\Omega  \abs{P(x, D) \ff} \abs{\ff(x)} e^{2 s \rho (x)} dx
+  \int_\Omega  \abs{\gamma_0(x)\zeta(x)} \abs{\ff(x)}^2  e^{2 s \rho (x)} dx\cr
&\leq& C \int_\Omega  \abs{P(x, D) \ff} \abs{\ff(x)} e^{2 s \rho (x)} dx
+ C \int_\Omega \abs{\ff(x)}^2  e^{2 s \rho (x)} dx.
\end{eqnarray*}
Then for large $s$, we get
\begin{eqnarray*}
 s \int_\Omega \abs{\ff(x)}^2  e^{2 s \rho (x)} dx
&\leq& C \int_\Omega  \abs{P(x, D) \ff} \abs{\ff(x)} e^{2 s \rho (x)} dx.
\end{eqnarray*}
On the other hand, for all small $\varepsilon >0$ there exist a constant $C_\varepsilon>0$ such that
 \begin{multline}\label{r5}
\int_\Omega  \abs{P(x, D) \ff \ff(x) \zeta(x) }e^{2 s \rho (x)} dx \leq 
\frac{C_\varepsilon}{s}\int_\Omega  \abs{P(x, D) \ff}^2 e^{2 s \rho (x)} dx
+ \varepsilon s \int_\Omega \abs{\ff(x)}^2  e^{2 s \rho (x)} dx.
\end{multline}
In terms of (\ref{r5}), we have
\begin{equation*}
s \int_\Omega \abs{\ff(x)}^2  e^{2 s \rho (x)} dx \leq  \frac{C_\varepsilon}{s}\int_\Omega  \abs{P(x, D) \ff}^2 e^{2 s \rho (x)} dx
+ \varepsilon s \int_\Omega \abs{\ff(x)}^2  e^{2 s \rho (x)} dx.
\end{equation*}
Moreover, for all small $\varepsilon>0$, we get
\begin{equation*}
s \int_\Omega \abs{\ff(x)}^2  e^{2 s \rho (x)} dx
 \leq  \frac{C_\varepsilon}{s}\int_\Omega  \abs{P(x, D) \ff}^2 e^{2 s \rho (x)} dx.
\end{equation*}
Consequently,
\begin{equation*}
s^2 \int_\Omega \abs{\ff(x)}^2  e^{2 s \rho (x)} dx
 \leq C_\varepsilon\int_\Omega  \abs{P(x, D) \ff}^2 e^{2 s \rho (x)} dx.
\end{equation*}

Since $$ P(x,D) \p_j\ff = \p_j P(x,D)\ff - (\p_j P(x,D))\ff,\qquad \p_j \ff\mid_{\Gamma}=0, $$
we apply (\ref{6.6}) to $\p_j \ff$, for $j= 1, 2, 3$, we get
\begin{multline*}
s^2 \int_\Omega   e^{2s \rho (x)} \abs{\p_j\ff}^2 dx
\leq  C \int_\Omega e^{2 s \varphi(x, t)}\abs{P(x,D) \p_j\ff}^2 dx dt\cr
\leq C \int_\Omega e^{2 s \varphi(x, t)}\abs{\p_j P(x,D)\ff - (\p_j P(x,D))\ff}^2 dx dt\cr
\leq C \int_\Omega e^{2 s \varphi(x, t)}(\abs{\ff}^2 + \abs{\p_j\ff}^2) dx dt
+ C \int_\Omega e^{2 s \varphi(x, t)} \abs{\p_j P(x,D)\ff}^2  dx dt\cr
\leq C \int_\Omega e^{2 s \varphi(x, t)}(\abs{P(x,D)\ff}^2 + \abs{\p_j P(x,D)\ff}^2) dx dt
+ C \int_\Omega e^{2 s \varphi(x, t)}\abs{\nabla\ff}^2 dx dt.
\end{multline*}
Therefore,
\begin{multline*}
s^2 \int_\Omega  e^{2s \rho (x)} \abs{\nabla\ff}^2 dx
\leq   C \int_\Omega e^{2 s \varphi(x, t)}(\abs{P(x,D)\ff}^2+ \abs{\nabla P(x,D)\ff}^2) dx dt\cr
 + C \int_\Omega e^{2 s \varphi(x, t)}\abs{\nabla\ff}^2 dx dt.
\end{multline*}
For sufficiently large $s$, we can complete the proof of the lemma \ref{L7}.

\end{proof}

The second one, is
\begin{lemma}\label{L6}{}
Let $( \vv^i, y^i)$ satisfy
\begin{equation}\label{8.5}
\left\{
\begin{array}{lll}
 \vv^i_{tt}- \Delta_{\mu, \lambda}\vv^i -  \nabla(\lambda^* \dive \vv^i_t) + \varrho_1 \nabla y^i = \ff^i &\,\,
(x,t) \in Q,\cr
 y^i_t - \Delta y^i + \varrho_2 \dive \vv^i_t= g^i & \,\,
(x,t) \in Q,\cr
\vv^i = 0,\qquad y^i =0 &\,\,
(x,t) \in \Sigma.
\end{array}
\right.
\end{equation}
 Then there exist constants $C>0$, $s_*>0$, for $i=1,2,3$, such that
\begin{multline}\label{6.0}
s \int_\Omega  e^{2s \varphi(x, t_0)} \abs{y^i(x, t_0)}^2 dx \leq  \int_Q e^{2 s \varphi}\abs{g^i(x,t)}^2 dxdt
+ C N_\sigma (\widetilde{\vv^i}, \widetilde{y^i})\cr
+  s e^{2 s d_0} \para{\norm{\vv^i}_{H^1(0, T; H^1(\Omega))}^2 + \norm{y^i}_{L^2(Q)}^2},
\end{multline}
\begin{equation}\label{6.5}
s^2 \int_\Omega e^{2 s \varphi(x, t_0)}\abs{\vv^i(x, t_0)}^2dx
\leq C \para{ N_\sigma (\widetilde{\vv}^i, \widetilde{y}^i)+ s^{-2} N_\sigma (\widetilde{\vv_t}^i, \widetilde{y_t}^i)}.
\end{equation}
\begin{equation}\label{6.3}
s^2 \int_\Omega  e^{2s \varphi(x, t_0)} \abs{\dive \vv^i(x, t_0)}^2 dx \leq
 C \para{ N_\sigma (\widetilde{\vv^i}, \widetilde{y^i})+ s^{-2} N_\sigma (\widetilde{\vv_t^i}, \widetilde{y_t^i})}
\end{equation}
 and
 \begin{equation}\label{6.4}
\int_\Omega  e^{2s \varphi(x, t_0)} \abs{\nabla\dive \vv^i(x, t_0)}^2 dx \leq C N_\sigma (\widetilde{\vv^i}, \widetilde{y^i}),
\end{equation}
for all $s \geq s_*$.
\end{lemma}
\begin{proof}
Applying Lemma \ref{L4} to $z^i(x, t) = \eta(t)  e^{2 s\varphi(x,t)} y^i(x,t)  =e^{2 s\varphi(x,t)} \widetilde{y^i}(x,t)$ and by the second
equation of (\ref{8.5}), we obtain
\begin{multline*}
\int_\Omega e^{2 s \varphi(x, t_0)}\abs{y^i(x, t_0)}^2dx \leq s \int_Q  e^{2 s\varphi}\abs{\widetilde{y}^i(x,t)}^2dxdt + s^{-1} \int_Q  e^{2 s \varphi}\abs{\widetilde{y_t}^i(x, t)}^2dxdt \cr
\leq s\int_Q  e^{2 s\varphi}\abs{\widetilde{y}^i(x,t)}^2dxdt + s^{-1}\int_Q  e^{2 s \varphi}\abs{g^i(x,t)}^2 dxdt+  s^{-1}\int_Q  e^{2 s \varphi}\abs{\Delta \widetilde{y}^i(x,t)}^2 dxdt\cr + s^{-1} \int_Q  e^{2 s \varphi}\abs{\dive \widetilde{\vv^i_t}}^2 dxdt+ e^{2 s d_0} (\norm{\vv^i}^2_{H^1(0,T;H^1(\Omega))}+ \norm{y^i}^2_{L^2(Q)}).
\end{multline*}
Then,
\begin{multline*}
Cs \int_\Omega  e^{2s \varphi(x, t_0)} \abs{y^i(x, t_0)}^2 dx \leq  \int_Q e^{2 s \varphi}\abs{g^i(x,t)}^2 dxdt\cr
+  \int_Q e^{2 s \varphi} \Big( \abs{\Delta \widetilde{y}^i}^2  + s^2 \abs{\widetilde{y}^i}^2 + \abs{\dive \widetilde{\vv_t}^i}^2\Big) dxdt
+  s e^{2 s d_0} (\norm{\vv^i}_{H^1(0, T; H^1(\Omega))}^2 + \norm{y^i}_{L^2(Q)}^2)\cr
\leq  \int_Q e^{2 s \varphi}\abs{g^i(x,t)}^2 dxdt
+ C N_\sigma (\widetilde{\vv^i}, \widetilde{y^i})
+  s e^{2 s d_0} (\norm{\vv^i}_{H^1(0, T; H^1(\Omega))}^2 + \norm{y^i}_{L^2(Q)}^2).
\end{multline*}

Applying Lemma \ref{L4} to $z^i(x, t) = \eta(t)  e^{2 s\varphi(x,t)} \vv^i(x,t)  =e^{2 s\varphi(x,t)}  \widetilde{\vv}^i(x,t)$, we obtain
\begin{multline*}
s^2 \int_\Omega e^{2 s \varphi(x, t_0)}\abs{\vv^i(x, t_0)}^2dx \leq s^3 \int_Q  e^{2 s\varphi}\abs{ \widetilde{\vv}^i(x,t)}^2dxdt + s\int_Q  e^{2 s \varphi}\abs{\widetilde{\vv_t}^i(x, t)}^2dxdt \cr
\leq C \para{ N_\sigma (\widetilde{\vv}^i, \widetilde{y}^i)+ s^{-2} N_\sigma (\widetilde{\vv_t}^i, \widetilde{y_t}^i)}.
\end{multline*}

Applying Lemma \ref{L4} to $z^i(x, t) = \eta(t)  e^{2 s\varphi(x,t)} \dive \vv^i(x,t)  =e^{2 s\varphi(x,t)} \dive \widetilde{\vv}^i(x,t)$, we obtain
\begin{multline*}
s^2 \int_\Omega e^{2 s \varphi(x, t_0)}\abs{\dive \vv^i(x, t_0)}^2dx \leq s^3 \int_Q  e^{2 s\varphi}\abs{\dive \widetilde{\vv}^i(x,t)}^2dxdt + s\int_Q  e^{2 s \varphi}\abs{\dive \widetilde{\vv_t}^i(x, t)}^2dxdt \cr
\leq C \para{ N_\sigma (\widetilde{\vv}^i, \widetilde{y}^i)+ s^{-2} N_\sigma (\widetilde{\vv_t}^i, \widetilde{y_t}^i)}.
\end{multline*}

Finally, applying Lemma \ref{L4} to $z^i(x, t) = \eta(t)  e^{2 s\varphi(x,t)} \nabla\dive \vv^i(x,t)  =e^{2 s\varphi(x,t)} \nabla\dive \widetilde{\vv}^i(x,t)$, we obtain
\begin{multline*}
\int_\Omega e^{2 s \varphi(x, t_0)}\abs{\nabla\dive \vv^i(x, t_0)}^2dx \leq s \int_Q  e^{2 s\varphi}\abs{\nabla\dive \widetilde{\vv}^i(x,t)}^2dxdt + s^{-1}\int_Q  e^{2 s \varphi}\abs{\nabla\dive \widetilde{\vv_t}^i(x, t)}^2dxdt\cr
\leq C N_\sigma (\widetilde{\vv}^i, \widetilde{y}^i).
\end{multline*}

This complete the proof of the Lemma.
\end{proof}

Now we complete the proof of Theorem \ref{T.2}, let $v=\dive\vv$. Then we have
\begin{equation}\label{e3}
\dive (p(x)\nabla\theta^*(x,t_0))=\para{v_{tt}-(2\mu+\lambda)\Delta v -\Delta(\lambda^*v_t)+ \dive (\varrho_1\nabla y)}(x,t_0)\equiv F(x,t_0).
\end{equation}
Note that by the assumption $(A.2)$, we have $p(x)=0$ and $\nabla p(x)=0$ on $\Gamma$ and
$$
\abs{\nabla\theta^*(x,t_0)}\geq\varepsilon.
$$
Furthermore the assumption (\ref{1.5}) imply that
\begin{multline}
\int_\Omega \abs{F}^2 e^{2s\varphi(x,t_0)}dx
\leq C\int_\Omega \abs{\dive \vv^2(x,t_0)}^2 e^{2s\varphi(x,t_0)}dx\cr
+ C\int_\Omega\abs{\Delta v(x,t_0)}^2 e^{2s\varphi(x,t_0)}dx
+ C\int_\Omega\abs{\Delta (\lambda^* v_t(x,t_0))}^2 e^{2s\varphi(x,t_0)}dx\cr
+C\int_\Omega \abs{\dive (\varrho_1\nabla y(x,t_0))}^2e^{2s\varphi(x,t_0)}dx.
\end{multline}
Using (\ref{6.3}) in the Lemma \ref{L6}, we get
\begin{multline}\label{9.0}
s^2 \int_\Omega \abs{F}^2 e^{2s\varphi(x,t_0)}dx
\leq C\para{N_\sigma (\widetilde{\vv}^2, \widetilde{y}^2) + s^{-2} N_\sigma (\widetilde{\vv}^3, \widetilde{y}^3)}\cr
+ Ce^{D s } \para{ \norm{\dive\vv(., t_0)}_{H^2(\Omega)}^{2} + \norm{\dive\vv_t(., t_0)}_{H^2(\Omega)}^{2} + \norm{y (., t_0)}_{H^2(\Omega)}^{2}  }.
\end{multline}
Similarly, we have
\begin{multline}
\int_\Omega \abs{\nabla F}^2e^{2s\varphi(x,t_0)}dx\leq
C\int_\Omega  \abs{\nabla \dive \vv^2(x,t_0)}^2e^{2s\varphi(x,t_0)}dx\cr
+ C e^{Ds} \para{ \norm{\dive\vv(., t_0)}_{H^3(\Omega)}^{2} + \norm{\dive\vv_t(., t_0)}_{H^3(\Omega)}^{2} + \norm{y (., t_0)}_{H^3(\Omega)}^{2}}.
\end{multline}
Using  (\ref{6.4}) in the lemma \ref{L6}, we get
\begin{multline}\label{9.1}
\int_\Omega \abs{\nabla F}^2 e^{2s\varphi(x,t_0)}dx
\leq C N_\sigma (\widetilde{\vv}^2, \widetilde{y}^2)
+ C e^{Ds} \para{ \norm{\dive\vv(., t_0)}_{H^3(\Omega)}^{2} + \norm{\dive\vv_t(., t_0)}_{H^3(\Omega)}^{2} + \norm{y (., t_0)}_{H^3(\Omega)}^{2}}.
\end{multline}
By the second equation of (\ref{8.3}), we have
\begin{equation}\label{8.8}
\begin{array}{lll}
y_t(x,t_0) = q(x) \dive \uu^*_t(x, t_0) + \Delta y (x,t_0) - \varrho_2 \dive \vv_t (x,t_0) \textrm{in }\,\,
\Omega
\end{array}
\end{equation}
Then,
\begin{equation}
C\abs{q(x)} \leq \abs{q(x)} \abs{\dive \uu^*_t(x, t_0)}
 \leq \abs{  \Delta y (x,t_0)} + \abs{\dive \vv^1(x,t_0)} + \abs{ y^1(x,t_0)}.
\end{equation}
Moreover,
\begin{multline}
Cs \int_\Omega  e^{2s \varphi(x, t_0)} \abs{q(x)}^2 dx \leq
s \int_\Omega  e^{2s \varphi(x, t_0)}( \abs{  \Delta y (x,t_0)}^2 + \abs{\dive \vv^1(x,t_0)}^2 + \abs{ y^1(x,t_0)}^2) dx\cr
\leq s \int_\Omega  e^{2s \varphi(x, t_0)}  \abs{ y^1(x,t_0)}^2dx +
s \int_\Omega  e^{2s \varphi(x, t_0)} \abs{\dive \vv^1(x,t_0)} ^2dx
+   e^{Ds} \norm{y(.,t_0)}_{H^2(\Omega)}^{2}
\end{multline}
Then, using Lemma \ref{L6}, we get
\begin{multline}\label{9.2}
s \int_\Omega  e^{2s \varphi(x, t_0)} \abs{q(x)}^2 dx
\leq C \para{N_\sigma (\widetilde{\vv}^1, \widetilde{y}^1) + s^{-2} N_\sigma(\widetilde{\vv}^2, \widetilde{y}^2)+ s^{-4} N_\sigma (\widetilde{\vv}^3, \widetilde{y}^3)}   \cr
+ C s e^{2 s d_0} (\norm{\vv^1}_{H^1(0, T; H^1(\Omega))}^2) + \norm{y^1}_{L^2(Q)}^2) + e^{Ds}\norm{y(.,t_0)}_{H^2(\Omega)}^{2}.
\end{multline}
Then again, by (\ref{e3}), we observe that
$$F(x,t_0) = \nabla p(x)\cdot\nabla\theta^*(x,t_0) + p(x)\Delta\theta^*(x,t_0),\quad x\in\Omega .$$
So, we need Carleman estimate for the first-order partial differential equation given by  Lemma \ref{L7},
then we have
\begin{equation}\label{e4}
s^2 \int_\Omega  e^{2s \varphi(x, t_0)} \abs{p(x)}^2 dx \leq  C \int_\Omega e^{2 s \varphi(x, t)}\abs{F(x,t)}^2 dx dt,
\end{equation}
 and
\begin{equation}\label{e5}
s^2 \int_\Omega  e^{2s \varphi(x, t_0)} \abs{\nabla p(x)}^2 dx \leq
 C \int_\Omega e^{2 s \varphi(x, t)}(\abs{F}^2 + \abs{\nabla F}^2 ) dx dt.
\end{equation}
Moreover, using (\ref{9.0}), (\ref{9.1}) and (\ref{9.2})
\begin{multline*}\label{9.3}
\int_\Omega  e^{2s \varphi(x, t_0)}\para{ s^4 \abs{p(x)}^2 +   s^2\abs{\nabla p(x)}^2 + s \abs{q(x)}^2} dx\cr
\leq \int_\Omega  e^{2s \varphi(x, t_0)}\para{ s^2 \abs{F(x)}^2 +  \abs{\nabla F(x)}^2 + s \abs{q(x)}^2} dx\cr
\leq C \para{N_\sigma (\widetilde{\vv}^1, \widetilde{y}^1)+ N_\sigma (\widetilde{\vv}^2, \widetilde{y}^2)+ s^{-2} N_\sigma (\widetilde{\vv}^3, \widetilde{y}^3)} \cr
+ Ce^{Ds} \para{ \norm{\dive\vv(., t_0)}_{H^3(\Omega)}^{2} + \norm{\dive\vv_t(., t_0)}_{H^3(\Omega)}^{2} + \norm{y (., t_0)}_{H^3(\Omega)}^{2}}
+ C s e^{2 s d_0} M^2_0.
\end{multline*}
Furthermore,
\begin{multline*}
\int_\Omega  e^{2s \varphi(x, t_0)}\para{ s^4 \abs{p(x)}^2 + s^2\abs{\nabla p(x)}^2 + s \abs{q(x)}^2} dx\leq
C \int_Q s\abs{q}^2e^{2s\varphi}dxdt\cr
+ C e^{Ds} \displaystyle\sum_{i=1}^{3}\para{\norm{\vv^i}^2_{H^4(\omega\times(0,T))}
+\norm{\dive\vv^i(\cdot,t_0)}_{H^3(\Omega)}^2+
\norm{\dive\vv^i_t(\cdot,t_0)}_{H^3(\Omega)}^2+\norm{y(.,t_0)}_{H^3(\Omega)}^{2}}\cr
+ Cs^2e^{2d_0s}M^2_0
\end{multline*}
for all large $s>0$.\\

Using the Lebesgue theorem, we get as $s \rightarrow \infty$
\begin{equation*}
 C s \int_\Omega \abs{q}^2e^{2s\varphi(x,t_0)} \int_{0}^{T }e^{2s(\varphi(x,t)- \varphi(x,t_0))}dt dx
\leq C\, s\; o(1)\int_\Omega \abs{q}^2e^{2s\varphi(x,t_0)} dx.
\end{equation*}
Then, we get
\begin{multline*}
\int_\Omega  e^{2s \varphi(x, t_0)}\para{ s^4 \abs{p(x)}^2 + s^2\abs{\nabla p(x)}^2 + s \abs{q(x)}^2} dx\cr
\leq
C e^{Ds} \para{\norm{\vv}^2_{H^7(\omega\times (0, T))}+\norm{y (., t_0)}_{H^3(\Omega)}^{2}+\sum_{i=1}^{4}\norm{\dive\p_t^i\vv(\cdot,t_0)}_{H^3(\Omega)}^2}
+ Cs^2e^{2d_0s}M^2_0,
\end{multline*}
for all $s\geq s_*$.\\
Using (\ref{2.7}) and (\ref{2.8}), we obtain
\begin{multline*}
e^{2s d }\int_\Omega\para{  \abs{p(x)}^2 + \abs{\nabla p(x)}^2 +  \abs{q(x)}^2}dx
\leq C e^{Ds} (\norm{\vv}^2_{H^7(\omega\times (0, T))} + M_{t_0}(\vv, y))+ C s^2 e^{2d_0s} M_0^2,
\end{multline*}
for all $s \geq s_*$.\\
Moreover,
\begin{multline}
 \int_\Omega\para{ \abs{p(x)}^2 + \abs{\nabla p(x)}^2 + \abs{q(x)}^2}dx\cr
\leq C e^{(D - 2 d)s} (\norm{\vv}^2_{H^7(\omega\times (0, T))}+  M_{t_0}(\vv, y))+ C e^{2(d_0- d)s} s^2 M_0^2,
\end{multline}
for all $s \geq s_*$.\\
Besides,
\begin{multline}
 \int_\Omega\para{\abs{p(x)}^2+\abs{\nabla p(x)}^2+\abs{\Delta p(x)}^2}dx\cr
\leq C e^{D_0 s} (\norm{\vv}^2_{H^7(\omega\times (0, T))}+ M_{t_0}(\vv, y))
+ C e^{-2(d- d_0)s} M_0^2,
\end{multline}
for all $s \geq s_*$.\\
Finally, minimizing the right hand side with respect to $s$, we obtain:  there exist $\delta\in (0, 1)$ such that
\begin{equation*}
 \norm{p}^2_{H^1(\Omega)} + \norm{q}^2_{L^2(\Omega)}
\leq C \para{\norm{\vv}^2_{H^7(\omega\times (0, T))}+ M_{t_0}(\vv, y)}^{\delta}.
\end{equation*}
The proof of Theorem \ref{T.2} is completed.


\end{document}